\documentclass[final,3p,times]{elsarticle}
\usepackage{amssymb}
 \usepackage{amsthm,amsmath}

\newtheorem{remark}{Remark}[section]

\newtheorem{theorem}{Theorem}[section]

\newtheorem{lema}{Lemma}[section]
\newtheorem{hypothesis}[]{H\hspace{-0.1cm}}

 \usepackage{lineno}

\journal{}

\begin{document}

\begin{frontmatter}

%% Title, authors and addresses

%% use the tnoteref command within \title for footnotes;
%% use the tnotetext command for theassociated footnote;
%% use the fnref command within \author or \address for footnotes;
%% use the fntext command for theassociated footnote;
%% use the corref command within \author for corresponding author footnotes;
%% use the cortext command for theassociated footnote;
%% use the ead command for the email address,
%% and the form \ead[url] for the home page:
%% \title{Title\tnoteref{label1}}
%% \tnotetext[label1]{}
%% \author{Name\corref{cor1}\fnref{label2}}
%% \ead{email address}
%% \ead[url]{home page}
%% \fntext[label2]{}
%% \cortext[cor1]{}
%% \address{Address\fnref{label3}}
%% \fntext[label3]{}
\title{A numerical study of third--order equation with time--dependent coefficients: KdVB equation}

%% use optional labels to link authors explicitly to addresses:
%% \author[label1,label2]{}
%% \address[label1]{}
%% \address[label2]{}

\author[label1]{Cristhian Montoya}
\author[label2]{Carlos Spa}

\address[label1]{School of Applied Sciences and Engineering, Universidad EAFIT.  Medell\'{i}n, Colombia.}
\address[label2]{CASE department in the Barcelona Supercomputing Center. Plaza Eusebi G\"{u}ell n3. Barcelona. Spain}

\begin{abstract}
In this article we present a numerical analysis for a third--order differential equation with
	non--periodic boundary conditions and time--dependent coefficients, namely, the linear Korteweg--de Vries Burgers 
	equation. This numerical analysis is motived due to the dispersive and dissipative phenomena that government this kind 
	of equations. This work builds on previous methods for dispersive equations with constant coefficients, expanding
	the field to include a new class of equations which until now have eluded the time--evolving parameters.  
	More precisely, throughout the Legendre--Petrov--Galerkin method  we prove stability and convergence results of the 
	approximation in appropriate weighted Sobolev spaces. These results allow to show the role and trade off 
	of these temporal parameters into the model.  Afterwards, we numerically investigate the dispersion--dissipation 
	relation for several profiles, further provide insights into the implementation method, which allow to exhibit the accuracy 
	and efficiency of our numerical algorithms.
\end{abstract}

\begin{keyword}
Korteweg--de Vries Burgers equation, Legendre–Petrov–Galerkin method, stability analysis, convergence, 
time--dependent coefficients
%% keywords here, in the form: keyword \sep keyword

%% PACS codes here, in the form: \PACS code \sep code

%% MSC codes here, in the form: \MSC code \sep code
%% or \MSC[2008] code \sep code (2000 is the default)

\end{keyword}

\end{frontmatter}

% \linenumbers

%% main text
\section{\normalsize Introduction} 

\quad

Since the Orszag's pioneer works in the early seventies \cite{1977Orszag},  several numerical spectral methods for 
	solving initial value problems of partial differential equations (PDEs) have become increasingly popular in recent years, 
	specially those associated to spectral Galerkin approximations. As it is well known, the spectral methods involve 
	representation the solutions as a truncated series of known functions of the independent variables. In addition, 
	due to the high--order accuracy whenever they work, these methods are preferable in numerical solutions of PDEs.
	In that framework, Jacobi polynomials \cite{bookGabor} have been used in a variety of applications due to their 
	ability to approximate general classes of functions as well as to its orthogonality properties; for instance, in resolution 
	of the Gibbs' phenomenon \cite{1997David}, transverse vibrations in beams and plates \cite{2007Caruntu}, 
	electrocardiogram data compresion \cite{philips1992data,tchiotsop2007ecg},
	and solution to even--order differential equations subject to various boundary conditions 
	\cite{2009Doha1,2009Doha3,2009Doha2}.  
	
	For the class of odd--order differential equations, which only includes a dispersive process, such as the 
	Korteweg--de Vries (KdV) equation  \cite{korteweg1895xli}, it is well known that periodic boundary conditions onto 
	those models allow for example to apply the Fourier spectral method for obtaining stability results and error estimates,
	 see for instance, Fornberg and Whitham \cite{1978Fornberg}, Fenton and Rienecker \cite{1982Fenton}, Maday and 
	 Quarteroni \cite{1988Quarteroni}, Deng and Ma \cite{2009DengMa}, and references therein.
	 Nevertheless, by considering a bounded domain with non--periodic boundary conditions on those models, 
	 polynomial spectral schemes based in Jacobi polynomials have been successfully used for their numerical 
	 approximations in space \cite{2012Bhrawy, 2000Maetal, 2001Maetal, 2000Li,2007Goubet,1994ShenJie,2010Korkmaz}.
	A common characteristic in the previous papers is the Petrov--Galerkin formulation, which appears due to the lack 
	of symmetry of the main operator and simultaneously incorporates  the boundary conditions inside the polynomial bases. 
	This fact allow to integrate by parts freely in space omitting any  additional boundary terms. In addition, 
	it is worth mentioning the hybrid method proposed by Ma and Sun \cite{2000Maetal}, where the linear part of the KdV equation
	was treated by a Legendre--Petrov--Galerkin (LPG) method, and the nonlinear term was treated using a Chebyshev--collocation 
	method. Moreover, they showed an estimate of order $N^{-r},\, r\geq 2$ in $L^2$--norm (i.e., $N$ is the number of modes) 
	for the linear KdV equation when the solutions satisfy a suitable regularity. Indeed, \cite{2000Maetal} constitutes a starting 
	point in our analysis.

	 By incorporating  a second--order term into the model, it might add a dissipative phenomenon such as occurs
	 in Burgers--type equations \cite{burgers1940application}. The resultant equation is the so--called 
	 Korteweg--de Vries--Burgers (KdVB) equation. In \cite{2003ShenJie}, Shen first introduced a dual LPG method 
	 for the KdVB equation, where the innovation lies in the choice of both trial and test spaces, which form a sequence of 
	 orthogonal polynomials in weighted Sobolev spaces. This fact allow to establish optimal error estimates in appropriate
	 Sobolev spaces. More precisely, in that paper, for the linear KdVB equation with constant coefficients, the author   
	 proven a rate of convergence of order $(1+|\beta|N)N^{1-r},\, r\geq 1$, where $\beta\in (-\frac{1}{3},\frac{1}{6})$ is 
	 the  coefficient associated to the  second--order term (in some cases it can represent the viscosity constant).
	 Afterwards, Yuan et al. \cite{2010Yuan} extended the method proposed by Shen to fifth--order KdV--type 
	 equations. On the other hand, by using Jacobi polynomials, Doha et al. \cite{2011Doha} proposed numerical schemes
	 for solving  both third and fifth--order differential equations with space--dependent coefficients, although no 
	 theoretical result have been provided.  
	 
	 Respect to the time discretization,  a classical Crank--Nicholson--leap--frog scheme 
	 describes a good convergence property for the nonlinear cases in many of the above papers, 
	 meanwhile, a forward Euler scheme is enough for the linear model. However, it should be pointed out that,
	 in all previous approaches, they have obtained either numerical or theoretical results  assuming always
	 that  both the dispersion and dissipation parameters are time--independent coefficients. 
	 
	In general aspects,  the ability to explicitly express the time--parametric dependence of coefficients in 
	a dynamic system is necessary for accurate and quantitive characterization of partial differential equations.
	In practice this is an important innovation since the parameters of physical systems often vary during the measurement 
	process, so that the parametric dependencies may be disambiguated from the model itself. Motivated by this, 
	our work builds on previous methods for dispersive equations with constant coefficients, expanding the field to include a new 
	class of equations which until now have eluded the time--evolving parameters. In fact, in contrast to the previous works, 
	in this paper we consider a third--oder equation with time--dependent coefficients, 
	namely, the KdVB equation with non--periodic boundary conditions. To the best of our knowledge, in the literature no rigorous 
	analysis of stability and 
	convergence of a numerical scheme exists with time--dependent coefficients for this model. The inclusion of dispersive and 
	dissipative parameters with temporal dependence makes  a more careful treatment of the numerical approximation, giving 
	now a first theoretical answer to the relation between them, as well as indicating appropriate ranges where the solutions
	can be tested in order to obtain desired rate of error. Thus our method takes as starting point the framework 
	proposed in \cite{2000Maetal}, in this way, the first goal of this paper is to construct the LPG scheme for 
	the KdVB equation with time varying coefficients. At this point, it is worth mentioning that under this setting, one main advantage
	is the use of   few number of modes for obtaining good numerical simulations. Additionally, our error analysis show that
	the convergence rate is suboptimal respect to diffusion parameters, while, for dispersion parameters, the convergence is optimal. 
	 It is consistent with the results above mentioned and studied in \cite{2000Maetal,2003ShenJie} for 
	constant coefficients. Nevertheless, due to the parametric dependency, our estimates are carried out by considering 
	an appropriate relation among the variable time coefficients, which in turn shows a strong correlation between 
	those coefficients and the LPG method.  
		 
	 The remainder of this paper is organized as follows. In Section \ref{section.teorica}, we develop a fully discrete approximation 
	 for the KdVB equation with time--varying coefficients, further, we prove stability and convergence estimates in weighted Sobolev 
	 spaces. In Section \ref{section.numerica}, we numerically investigate the dispersion--dissipation relation for several profiles, and 
	 providing  insights into the implementation method.  Finally, in the last section, we present the conclusions and  outlooks.
% -----------------------------------------
% -------THEORETICAL RESULTS----------------------------------------
%-----------------------------------------
\section{Method and theoretical results}\label{section.teorica}
	In this section we study the numerical approximation for third--order differential
	equations with non--periodic boundary conditions, which in turn include time--dependent coefficients. 
	Specifically, we consider equations of the form  
	\begin{equation}\label{sys.kdvb.linear}
	\left\{
    \begin{array}{llll}
    	\partial_t u+\alpha(t)\partial_x^3u-\beta(t)\partial_x^2u=f(x,t)&\mbox{in} &(-1,1)\times(0,T),\\
  		u(-1,t)=u(1,t)=\partial_x u(1,t)=0&\mbox{in} &(0,T),\\
		u(\cdot, 0)=u_0(\cdot)&\mbox{in}&(-1,1),
    \end{array}\right.
	\end{equation}
	where $u=u(x,t)$ is the state variable in $(-1,1)\times(0,T)$, $f=f(x,t)$ is an external force acting in the system, and 
	$u_0$ is the initial datum. System \eqref{sys.kdvb.linear} represents the so--called linear 
	Korteweg de Vries--Burgers (KdVB) equation. In 1895, Korteweg and de Vries developed an evolutionary model to 
	describe the propagation of long water waves  in channels of shallow depth, namely, KdV equation, in which two phenomena 
	are involved, dispersion (third-order derivative) and nonlinear convection (nonlinear term). The interaction of these 
	terms gives rise to a wave traveling at constant speed without losing its sharp, called 
	soliton \cite{bona1985travelling,feng2007traveling,korteweg1895xli}. As consequence of the union of the 
	KdV and Burgers equations arise the KdVB equation, which in our case has homogeneous non--periodic boundary 
	conditions. 
	
	From a numerical point of view, KdVB--type equations with constant coefficients have been widely studied by
	means of different methods; see for instance the works 	
	 \cite{2020QinMa, 2019WangPengzhan, 2018Fangetal, 2010Korkmaz,2000Maetal, 2001Maetal} and references therein
	 for more details. 
	 
	 On the other hand, by introducing variable coefficients
	$\alpha(t)>0$ and $\beta(t)>0$, the KdVB equation \eqref{sys.kdvb.linear} is useful to describe solitonic propagation 
	in fluids \cite{yu2011solitonic}, a variety of cosmic plasma phenomena 
	\cite{gao2015variety, 2019Liumulti,2014Triki, 2019Gao}, among others. Motivated by those applications and as mentioned, 
	as far as we know, an exhaustive numerical study for \eqref{sys.kdvb.linear} has not been reported. Therefore, 
	our article fill this gap giving theoretical--numerical answers to the relation between them, as well as indicating 
	appropriate ranges where the solutions can be tested. 
	%\vskip 0.5cm	

% -----------------------------------------
% -------LPG METHOD----------------------------------------
%-----------------------------------------
\subsection{Legendre--Petrov--Galerkin method.}		
	In this subsection we formulate a fully discrete finite element scheme based in the LPG method for solving 
	\eqref{sys.kdvb.linear}. We begin by describing the LPG approximation framework and listing the 
	basic properties used in the analysis. 
	
	For any real constants $a,\,b$, let $\omega^{a,b}(x)=(1-x)^{a}(1+x)^{b}$ be weight functions on $(-1,1)=:I$ . 
	The inner product and norm in $L^2_{\omega^{a,b}}(I)$ are 
	denoted by $(\cdot,\cdot)_{\omega^{a,b}}$ and $\|\cdot\|_{\omega^{a,b}}$ respectively. We will omit the subscript 
	$\omega^{a,b}$ whenever $\omega^{a,b}(x)=1$. Let $\mathbb{P}_N(I)$ be the space of polynomials of degree at 
	most $N$ on the interval $I$ and 
	\[V_N=\mathbb{P}_N(I)\cap H_0^{2,1}(I),\quad W_N=\mathbb{P}_N(I)\cap H_0^1(I),\]
	where $H_0^{2,1}(I)=\{v\in H^2(I)\cap H_0^1(I): \partial_x v(1)=0\}$.  
	
	Let $L_k$ be the $k$th degree Legendre polynomial. To continuation we recall some properties of 
	Legendre polynomials which will be used in this paper  (see \cite{bookGabor}).
	\begin{equation}\label{prop.legrendre1}
		\int\limits_{-1}^1 L_j(x)L_k(x)dx=\frac{2}{2k+1}\delta_{jk}.
	\end{equation}
	\begin{equation}\label{prop.legrendre2}
		L_n(x)=\frac{1}{2n+1}(L_{n+1}'(x)-L_{n-1}'(x)),\quad n\in\mathbb{N}.
	\end{equation}
	\begin{equation}\label{prop.legrendre3}
		L'_{n}(x)=\sum_{\substack{k=0\\ k+n\,\,\text{odd}}}^{n-1}(2k+1)L_k(x).
	\end{equation}
	Bonnet's recursion formula:
	\begin{equation}\label{prop.legrendre4}
		(2n+1)xL_n(x)=(n+1)L_{n+1}(x)+ nL_{n-1}(x),\quad n\in\mathbb{N}.	
	\end{equation}
	
	Now, the semidiscrete LPG method for \eqref{sys.kdvb.linear} consists in finding $u_N(t)\in V_N$ such that for almost 
	every $t\in (0,T)$
	\begin{equation}\label{sys.full.semidiscrete}
	\left\{
    \begin{array}{llll}
    	(\partial_t u_N(t),v)+\alpha(t)(\partial_x^3u_N(t),v)-\beta(t)(\partial_x^2u_N(t),v)
    	=(f(t),v)
    	& & \forall v\in W_{N-1},\\
		(u_N(0),v)=(u_0,v)& &\forall v\in W_{N-1},
    \end{array}\right.
	\end{equation}
	holds. 

	From \cite{2000Maetal}, we shall use appropriate basis functions such that the corresponding matrices are sparse. 
	To this end, for $N\geq 3$ and $n=\overline{[0,N-3]}$, we define the basis functions $\{\phi_n\}$
	for the space $W_{N-1}$ by:
	\[\phi_n(x)=c_{n+1}(L_n(x)-L_{n+2}(x)),\quad c_n=\displaystyle\frac{1}{2n+1}.\]
	
	By taking into account \eqref{prop.legrendre2}, it is easy to verify that 
	$\partial_x\phi(x)=-L_{n+1}(x)$. 
	
	Next, we introduce the semidiscrete state variable $u_N(\cdot, t)\in V_N$ on spectral space and its vector
	representation:
	\begin{equation*}\label{id.semidiscretestate}
		u_N(\cdot,t)=(1-x)\sum\limits_{n=0}^{N-3}\hat{u}_n(t)\phi_n(\cdot),\quad 	
		{\bf{u}}(t)=(\hat{u}_0(t),\dots, \hat{u}_{N-3}(t))^T.
	\end{equation*}
	
	Respect to  the time discretization, a classical forward scheme is considered. Thus, let $\Delta t$ be the step size in time
	space and $t_k=k\Delta t$ ($k=\overline{[0,n_T]}$ and $t_{N_T}=T=n_T\Delta t$). 
	
	\noindent By simplicity, $v^k(x):=v(x,t_k)$ is denoted
	by $v^k$ and 
	\[v^{k+\frac{1}{2}}=\frac{1}{2}(v^{k+1}+v^k).\] 
	
	The fully discrete spectral method for \eqref{sys.kdvb.linear} reads: to find $u^k_N\in V_N$ such that
	\begin{equation}\label{sys.full.discrete}
	\left\{
    \begin{array}{llll}
    	(\Delta t)^{-1}(u_N^{k+1}-u_N^{k},v)+\alpha(t_k)(\partial_x^3u_N^{k+\frac{1}{2}},v)
    	-\beta(t_k)(\partial_x^2u_N^{k+\frac{1}{2}},v)=(f^{k+\frac{1}{2}},v)
    	& & \forall v\in W_{N-1},\\
		(u_N^0,v)=(u_0,v)& &\forall v\in W_{N-1},
    \end{array}\right.
	\end{equation}
	for $k=\overline{[0,n_T-1]}$.
%\vskip 0.5cm	
% -----------------------------------------
% -------LPG METHOD--STABILITY ----------------------------------------
%-----------------------------------------	
\subsection{Stability analysis.} 
	In this subsection we establish  the stability of the Legendre--Petrov Galerkin method for solving 
	 the discrete system \eqref{sys.full.discrete}. Before that, some theoretical properties in weighted spaces must 
	 be considered.
	 
	 The following lemma establishes a Poincar\'e--type  inequality  \cite{2003ShenJie}.
	\begin{lema}\label{lema.poincare}
		Let $u\in V_N$. Then,
		\[\int\limits_I\frac{|u(x)|^2}{(1-x)^3}dx\leq \int\limits_I\frac{|\partial_x u(x)|^2}{1-x}dx.\]	
	\end{lema}
	\begin{remark}\label{obs.identities.aux}
		Observe that, for any $v\in W_{N-1}$ and $u\in V_N$ follows
 		\[-(\partial_x^2((1-x)v),\partial_x v)=2\|\partial_x v\|^2-((1-x)\partial_x^2 v,\partial_x v)
 		=\frac{3}{2}\|\partial_x v\|^2+|\partial_x v(-1)|^2\]
 		and 
 		\[-(\partial_x^2 u,(1-x)^{-1}u)=((1-x)^{-2}u,\partial_x u)+((1-x)^{-1}\partial_x u,\partial_x u)
 		=-\|u\|^2_{\omega^{-3,0}}+\|\partial_x u\|^2_{\omega^{-1,0}}\geq 0,\]
 		where the above inequality is consequence of Lemma \ref{lema.poincare}.	
	\end{remark}
	The next theorem provides expressions for the stability of the LPG scheme given in \eqref{sys.full.discrete}.  
	\begin{theorem}\label{th1.stability} 
		Let $T>0$ and $\omega(x)=\omega^{-1,0}(x)$. Then,
		\begin{enumerate}
		\item [i)]  Assume $\beta(t)\geq 0$ and $\alpha(t)\geq\displaystyle\frac{1}{3\ell}$,  for any  
			$\ell\geq 1$ and $ 0\leq t\leq T$. Then, the  LPG 
			approximation $u_N^k$ of \eqref{sys.kdvb.linear} satisfies 
			\begin{equation}\label{ine.stability1}
			\begin{array}{ll}
			&\|u_N^n\|_{\omega}^2+\Delta t\sum\limits_{k=0}^{n-1}(3\ell\alpha(t_k)-1)\|\partial_x u_N^{k+\frac{1}{2}}\|^2
			+\Delta t\sum\limits_{k=0}^{n-1}\alpha(t_k)|u_N^{k+\frac{1}{2}}(-1)|^2\\
			&\hspace{3cm}+\Delta t\sum\limits_{k=0}^{n-1}\beta(t_k)\Bigl(\|\partial_x u_N^{k+\frac{1}{2}}\|_{\omega}^2
			-\|u_N^{k+\frac{1}{2}}\|_{\omega^{-3,0}}^2\Bigr)\\
			&\leq 8\ell\|u_N^0\|^2_{\omega}+8\Delta t\ell^2\sum\limits_{k=0}^{n-1}\|f^{k+\frac{1}{2}}\|_{H^{-1}(I)}^2,	
			\end{array}
			\end{equation}
			for all $0<n\leq n_T$.	
		\item [ii)] Assume $\beta(t)>0$ and $\alpha(t)>0$, for $t\in [0,T]$. Then, the LPG 
			approximation $u_N^k$ of \eqref{sys.kdvb.linear} satisfies		
			\begin{equation}\label{ine.stability2}
			\begin{array}{ll}
			&\|u_N^n\|_{\omega}^2+\Delta t\sum\limits_{k=0}^{n-1}\alpha(t_k)\|\partial_x u_N^{k+\frac{1}{2}}\|^2
			+\Delta t\sum\limits_{k=0}^{n-1}\alpha(t_k)|u_N^{k+\frac{1}{2}}(-1)|^2\\
			&\hspace{3cm}+\Delta t\sum\limits_{k=0}^{n-1}\beta(t_k)\Bigl(\|\partial_x u_N^{k+\frac{1}{2}}\|_{\omega}^2
			-\|u_N^{k+\frac{1}{2}}\|_{\omega^{-3,0}}^2\Bigr)\\
			&\leq 4\|u_N^0\|^2_{\omega}+4\Delta t\sum\limits_{k=0}^{n-1}\displaystyle\frac{1}{\alpha(t_k)}\|f^{k+\frac{1}
			{2}}\|_{H^{-1}(I)}^2,	
			\end{array}
			\end{equation}
			for all $0<n\leq n_T$.		
		\end{enumerate}
	\end{theorem}
 	\begin{proof}[Proof of Theorem \ref{th1.stability}] \smallskip\
		\begin{enumerate}
 		\item  [i)] Let us consider the test function $v=2\ell v_N^{k+\frac{1}{2}}$ in \eqref{sys.full.discrete}, with
 			$v_N^k=u_N^k\omega$. Then, after using remark \ref{obs.identities.aux} we get
 			\begin{equation}\label{eq.aux1.proofth1}
 			\begin{array}{lll}
 				&2\ell(\Delta t)^{-1}(u_N^{k+1}-u_N^{k},u_N^k\omega)+3\ell\alpha(t_k)\|\partial_x v^{k+\frac{1}{2}}\|^2
 				+2\ell\alpha(t_k)|\partial_x v^{k+\frac{1}{2}}(-1)|^2\\
 				&\hspace{1cm}+2\ell\beta(t_k)\Bigl(-\|u_N^{k+\frac{1}{2}}\|^2_{\omega^{-3,0}}
 				+\|\partial_x u_N^{k+\frac{1}{2}}\|^2_{\omega}\Bigl)\\
 				&\leq 2|(f^{k+\frac{1}{2}}, \ell v_N^{k+\frac{1}{2}})|.
 			\end{array}
 			\end{equation}
 			Using the fact that $\|u\|_{H^{-1}(I)}=\sup\limits_{v\in H_0^1(I)}\frac{|(u,v)|}{|v|_2}$ and  
 			Young's inequality  (i.e., $ab\leq \frac{a^p}{p}+\frac{b^q}{q}$; $\frac{1}{p}+\frac{1}{q}=1$; $a,b\geq 0$)
 			with $p=q=2$ and $a=\ell |f^{k+\frac{1}{2}}|$, $b=|v_N^{k+\frac{1}{2}}|$, the right hand side of the previuos
 			estimate can be estimated by 
 			\[2|(f^{k+\frac{1}{2}}, \ell v_N^{k+\frac{1}{2}})|
 			\leq \ell^2\|f^{k+\frac{1}{2}}|\|_{H^{-1}(I)}^2+\|\partial_ x v_N^{k+\frac{1}{2}}\|^2.\]
 			Replacing the above estimate into \eqref{eq.aux1.proofth1} and summing for $k=[\overline{0,n-1}]$, we obtain 
 			\begin{equation*}
 			\begin{array}{ll}
				&\ell\|u_N^n\|_{\omega}^2+\Delta t\sum\limits_{k=0}^{n-1}(3\ell\alpha(t_k)-1)\|\partial_x v_N^{k+\frac{1}{2}}\|^2
				+2\Delta t\ell\sum\limits_{k=0}^{n-1}\alpha(t_k)|v_N^{k+\frac{1}{2}}(-1)|^2\\
				&\hspace{1cm} +2\Delta t\ell\sum\limits_{k=0}^{n-1}\beta(t_k)\Bigl(-\|u_N^{k+\frac{1}{2}}\|^2_{\omega^{-3,0}}
 				+\|\partial_x u_N^{k+\frac{1}{2}}\|^2_{\omega}\Bigl)\\
				&\leq \ell\|u_N^0\|^2_{\omega}+\Delta t\ell^2\sum\limits_{k=0}^{n-1}\|f^{k+\frac{1}{2}}\|_{H^{-1}(I)}^2,	
			\end{array}
			\end{equation*}
			Finally, note that $\partial_x u_N^{k+\frac{1}{2}}(-1)=2\partial_x v_N^{k+\frac{1}{2}}(-1)$ and 
			\[\|\partial_x u_N^{k+\frac{1}{2}}\|^2\leq 2\|(1-x)\partial_x v_N^{k+\frac{1}{2}}\|^2
			+2\|\partial_x v_N^{k+\frac{1}{2}}\|^2\leq 8\|\partial_x v_N^{k+\frac{1}{2}}\|^2.\]
			This argument allows to deduce \eqref{ine.stability1} and ends the proof to the first case.
			
		\item  [ii)] Since the proof of \eqref{ine.stability2} follows the above structure, we have omitted the details.
			However, in this case, the term $(f^{k+\frac{1}{2}}, v_N^{k+\frac{1}{2}})$ is upper bounded by 
			using Young's inequality
 			(i.e., $ab\leq \frac{a^p}{p}+\frac{b^q}{q}$; $\frac{1}{p}+\frac{1}{q}=1$; $a,b\geq 0$)
 			with $p=q=2$ and $a=\alpha^{\frac{1}{2}}|v_N^{k+\frac{1}{2}}|$,\, 
 			$b=\alpha^{-\frac{1}{2}}|f^{k+\frac{1}{2}}|$.   	
 		\end{enumerate}
		This concludes the proof of Theorem \ref{th1.stability}. 
 	\end{proof}
 	\begin{remark}\label{obs.estability1}
 		In the dissipative case (that is, $\alpha=0$) where the symmetry of the main operator is guaranteed, 
 		a Galerkin approximation with Legendre polynomials in $W_N$ turns out to be more convenient than the 
		Petrov--Galerkin method. Otherwise, a LPG scheme might be inestable even for certain constant dissipation coefficients. 
 	\end{remark}	
 	\begin{remark}
	 	Theorem \ref{th1.stability} allows to visualize how the presence of a positive time--dependent dispersion
	 	coefficient could affect the stability if an external source acts into the system, see 
	 	\eqref{ine.stability2}. Nevertheless, it can be corrected by considering a restriction upon the
	 	dispersion coefficient, and therefore the LPG method is stable, see \eqref{ine.stability1}.  
	 \end{remark} 	
% -----------------------------------------
% -------ERROR ANALYSIS----------------------------------------
%-----------------------------------------
\subsection{Error analysis}	 
	In this paragraph, we present approximation properties of some projection operators, which are used later on. 	
	First, we recall a basic result of Jacobi polynomial approximation \cite{2000Li}. Let $P_N^{a,b}$ be the 
	$L^2_{\omega^{a,b}}(I)$--orthogonal projector $L^2_{\omega^{a,b}}(I)\mapsto \mathbb{P}_N(I)$ and by simplicity, 
	$P_N:=P_N^{0,0}$. 
	\begin{lema}\label{lema.projector}
		Assume $a,b>1$ and $r>0$. Then, for any $v\in H^r(I)$ and any $0 \leq s\leq r$, 
		\begin{equation}\label{eq.projector}
			\|\partial_x^s(P_N^{a,b}v-v)\|_{\omega^{a+s,b+s}}\leq CN^{s-r}\|\partial_x^r v\|_{\omega^{a+r,b+r}}.
		\end{equation}
	\end{lema}
	Now, from \cite{2000Maetal}, we define $\Pi_N:H_0^{2,1}(I)\mapsto V_N$ such that 
	\begin{equation}\label{def.pi}
		(\partial_x^2(\Pi_N u-u),\partial_x v)=0,\quad \forall v\in W_{N-1}.
	\end{equation}
	Moreover, $\Pi_N$ satisfies 
	\begin{equation}\label{property0.pi}
		\Pi_N u:=\bar{\partial}_x^{-2}P_{N-2}\partial_{x}^{2}u,
	\end{equation}
	where 
	\[\bar{\partial}_x^{-1}v(x):=-\int\limits_{x}^{1}v(y)dy,\quad\quad \bar{\partial}_x^{-i}v(x)
	=(\bar{\partial}_x^{-1})^{i}v(x).\]	

	\begin{lema}\label{lema.properties.pi}
	Assume $u\in H_0^{2,1}(I)\cap H^r(I)$ and $r\geq 2$. Then
	\begin{enumerate}
		\item [i)] 	$\|\partial_x^s(\Pi_Nu-u)\|_{\omega^{s-2,s-2}}\leq CN^{s-r}\|\partial_x^r u\|_{\omega^{r-2,r-2}}$,\, for any $0\leq s\leq 2$.
		\item [ii)] $(\Pi_N u-u,v)=0,$  for all $v\in \mathbb{P}_{N-4}(I).$
		\item [iii)] $\|\partial_x(\Pi_Nu-u)\|\leq CN^{2-r}\|\partial_x^r u\|_{\omega^{r-2,r-2}}$. 
	\end{enumerate}
	\end{lema}
	\begin{proof} The proof of the first two items can be found in \cite{2000Li}. Thus, we only proof the last point. 
	
		Then, using Lemma \ref{lema.projector}, \eqref{property0.pi}, the first part of this lemma, and also the 
		Poincar\'e inequality, we can deduce
		\begin{equation*}
 		\begin{array}{ll}
			\|\partial_x(\Pi_Nu-u)\|^2&=(\partial_x(\Pi_Nu-u),\partial_x\bar\partial_x^{-1}(\partial_x(\Pi_Nu-u)))\\
			&=|(\partial_x^2(\Pi_Nu-u),\bar\partial^{-1}_x(\partial_x(\Pi_Nu-u)))|\\
			&=|((P_{N-2}-I)\partial_x^2 u,\bar\partial^{-1}_x(\partial_x(\Pi_Nu-u))|\\
			&=|((P_{N-2}-I)\partial_x^2 u,\Pi_Nu-u)|\\
			&\leq \|(P_{N-2}-I)\partial_x^2 u\| \|\Pi_Nu-u\|\\
			&\leq CN^{2-r}\|\partial_x^r u\|_{\omega^{r-2,r-2}}\|\partial_x(\Pi_Nu-u)\|.	
		\end{array}
		\end{equation*}		
	\end{proof}

	Now, for the error analysis, let $u_N^k$ be the numerical solution  of the scheme \eqref{sys.full.discrete}  and let 
	$u\in X$ be a solution associated to  \eqref{sys.kdvb.linear}, where 
	\begin{equation*}\label{p.regularity.solutions}
		X\equiv C([0,T];H_0^{2,1}(I)\cap H^r(I))\cap H^1(0,T;H_0^{2,1}(I)\cap H^{\max\{2,r-1\}}(I))\cap 
		H^3(0,T;H^{-1}(I)),	\,\,\, r\geq 2.
	\end{equation*}
	Moreover, an assumption on the  coefficients $\alpha,\beta\in L^\infty([0,T])$ must be imposed, namely: 
	\begin{hypothesis}\label{h.th.error}
		For every $\varepsilon_1,\varepsilon_2>0$, 	
		\begin{equation}\label{hypo.H.teoerror}
		\Bigl(\Bigl(\displaystyle\frac{3}{8}-\displaystyle\frac{1}
 		{16}\varepsilon_1\Bigr)\alpha^k
 		-\Bigl(\displaystyle\frac{1}{8}\varepsilon_2+\displaystyle\frac{9}{8}\Bigr)\beta^k\Bigr)>0,
 		\quad \forall k=\overline{[0,n_T-1]}.
		\end{equation}
	\end{hypothesis} 
	In what follows, $\hat{e}_N^k=\Pi_Nu(\cdot,t_k)-u_N^k$,\, $\tilde{e}^k_N=u(\cdot,t_k)-\Pi_Nu(\cdot,t_k)$,\, and 
	$e^k_N=u(\cdot,t_k)-u_N^k=\tilde{e}^k_N+\hat{e}_N^k$.	
	\begin{theorem}\label{teo.error}
		Let $u\in X$ and let H\ref{h.th.error} be satisfied. Then, for $0\leq n\leq n_T$ 
		\begin{equation}\label{ine.teo.error1}
			\|u_{N}^{n}-u^{n}\|\leq\sqrt{2} \|u_{N}^{n}-u^{n}\|_{\omega^{-1,0}}\lesssim C_{\alpha}(N^{-r}+(\Delta t)^{2})
			+C_{\beta}(N^{1-r}+(\Delta t)^{2}),
		\end{equation}
		where $C_\alpha\approx \Bigl(\min\limits_{t\in [0,T]}\alpha(t)\Bigr)^{-1}$ and 
		$C_\beta\approx \|\beta\|_{L^\infty([0,T])}$.
	\end{theorem}
	\begin{proof}[Proof of Theorem \ref{teo.error}]
		From \eqref{sys.kdvb.linear}, \eqref{sys.full.discrete} and \eqref{def.pi}, for $k=\overline{[0,n_T-1]}$ and
		for any $v\in W_{N-1}$ we get
		\begin{equation*}\label{sys.error}
		\left\{
    	\begin{array}{llll}
    		(\Delta t)^{-1}(\hat{e}_N^{k+1}-\hat{e}_N^{k},v)-\alpha^{k}(\partial_x^2\hat{e}_N^{k+\frac{1}{2}},\partial_x v)
    		-\beta^{k}(\partial_x^2\hat{e}_N^{k+\frac{1}{2}},v)=(g^{k},v)
    		+\beta^{k}(\partial_x^2\tilde{e}_N^{k+\frac{1}{2}},v)
    		-\beta^{k}(\partial_x^2 e_N^{k+\frac{1}{2}},v),
    		& &\\
			(\hat{e}_N^0,v)=(\Pi_Nu^0-u_N^0,v),& &
    	\end{array}\right.
		\end{equation*}
		where 
		\[g^{k}=\partial_t{u}^{k+\frac{1}{2}}-\Pi_N\partial_t{u}^{k}.\]
		By simplicity, let $\omega(x)=\omega^{-1,0}(x)$. Considering in the previous system the 
		test function $v=2\omega\hat{e}_N^{k+\frac{1}{2}}$, and using Remark \ref{obs.identities.aux}, we get
		\begin{equation}\label{eq.aux.errors}
 			\begin{array}{lll}
 				&(\Delta t)^{-1}(\|\hat{e}_N^{k+1}\|^2_{\omega}-\|\hat{e}_N^{k}\|^2_{\omega})
 				+3\alpha^{k}\|\partial_x(\omega\hat{e}_N^{k+\frac{1}{2}})\|^2
 				+2\alpha^{k}|\partial_x(\omega\hat{e}_N^{k+\frac{1}{2}})(-1)|^2\\
 				&\hspace{1cm}+2\beta^{k}\Bigl(-\|\hat{e}_N^{k+\frac{1}{2}}\|^2_{\omega^{-3,0}}
 				+\|\partial_x\hat{e}_N^{k+\frac{1}{2}}\|^2_{\omega}\Bigl)\\
 				&= 2(g^{k},\omega\hat{e}_N^{k+\frac{1}{2}})
 				+2\beta^{k}(\partial_x^2\tilde{e}_N^{k+\frac{1}{2}},\omega\hat{e}_N^{k+\frac{1}{2}})
 				+2\beta^k(\partial_x e_N^{k+\frac{1}{2}},\partial_x(\omega\hat{e}_N^{k+\frac{1}{2}})).
 			\end{array}
 		\end{equation}
 		
 		From Lemma \ref{lema.projector}, Lemma \ref{lema.properties.pi} and Young's inequality 
 		(i.e, $ab\leq \frac{a^2}{2\varepsilon}+\frac{\varepsilon b^2}{2}$,\, $\varepsilon>0$)
 		follows:
 		\begin{equation*}
 			\begin{array}{lll}
 				|(g^{k},\omega\hat{e}_N^{k+\frac{1}{2}})|
 				&\leq |(\partial_t{u}^{k+\frac{1}{2}}-\partial_t{u}^{k},\omega\hat{e}_N^{k+\frac{1}{2}})|
 				+|((I-\Pi_N)\partial_t{u}^k, (I-P_{N-4})\omega\hat{e}_N^{k+\frac{1}{2}})|\\ 
 				&\leq C\Bigl(\|\partial_t{u}^{k+\frac{1}{2}}-\partial_t{u}^k\|_{H^{-1}(I)} 
 				 +N^{-r}\|\partial_t{u}^k\|_{H^{\max\{2, r-1\}}(I)}\Bigr)
 				 \|\partial_x(\omega\hat{e}_N^{k+\frac{1}{2}})\|\\
 				&\leq \displaystyle\frac{\varepsilon_1\alpha^k}{2}\|\partial_x
 				(\omega\hat{e}_N^{k+\frac{1}{2}})\|^2
 				+C_{\alpha^k,\varepsilon_1}\Bigl(\|\partial_t{u}^{k+\frac{1}{2}}-\partial_t{u}^k\|^2_{H^{-1}(I)} 
 				+N^{-2r}\|\partial_t{u}^k\|^2_{H^{\max\{2, r-1\}}(I)}\Bigr), 
 			\end{array}
 		\end{equation*}
		where $C_{\alpha^k,\varepsilon_1}=\displaystyle\frac{C^2}{2\varepsilon_1\alpha^k}$, for every 
		$\varepsilon_1>0$.
 		
 		Observe that the last two terms in the right--hand side of \eqref{eq.aux.errors} can be upper bounded
 		by using again Young's inequalities. Thus, putting together those estimates, we obtain
 		\begin{equation}\label{}
 			\begin{array}{lll}
 				&(\Delta t)^{-1}(\|\hat{e}_N^{k+1}\|^2_{\omega}-\|\hat{e}_N^{k}\|^2_{\omega})
 				+3\alpha^{k}\|\partial_x(\omega\hat{e}_N^{k+\frac{1}{2}})\|^2
 				+2\alpha^{k}|\partial_x(\omega\hat{e}_N^{k+\frac{1}{2}})(-1)|^2\\
 				&\hspace{1cm}+2\beta^{k}\Bigl(-\|\hat{e}_N^{k+\frac{1}{2}}\|^2_{\omega^{-3,0}}
 				+\|\partial_x\hat{e}_N^{k+\frac{1}{2}}\|^2_{\omega}\Bigl)\\
 				&\leq \displaystyle\frac{\varepsilon_1\alpha^k}{2}\|\partial_x
 				(\omega\hat{e}_N^{k+\frac{1}{2}})\|^2
 				+C_{\alpha^k,\varepsilon_1}\Bigl(\|\partial_t{u}^{k+\frac{1}{2}}-\partial_t{u}^k\|^2_{H^{-1}(I)} 
 				+N^{-2r}\|\partial_t{u}^k\|^2_{H^{\max\{2, r-1\}}(I)}\Bigr)\\
 				&\quad\quad
 				+\varepsilon_2\beta^k\|\partial_x(\omega\hat{e}_N^{k+\frac{1}{2}})\|^2
 				+\displaystyle\frac{\beta^k}{\varepsilon_2}\|\partial_x\tilde{e}_N^{k+\frac{1}{2}}\|^2
 				+\beta^k\|\partial_x e_N^{k+\frac{1}{2}}\|^2
 				+\beta^k\|\partial_x(\omega\hat{e}_N^{k+\frac{1}{2}})\|^2,
 			\end{array}
 		\end{equation}
		for every $\varepsilon_1,\varepsilon_2>0$.
		
		Thus, assuming the relation 
		\[\Bigl(\Bigl(3-\frac{1}{2}\varepsilon_1\Bigr)\alpha^k-
		\Bigl(\varepsilon_2+1\Bigr)\beta^k\Bigr)>0\]
		and adding for $k=\overline{[0,n_T-1]}$,  we deduce 
 		\begin{equation*}
 			\begin{array}{lll}
 				&\|\hat{e}_N^n\|_{\omega}^2
 				+\Delta t\sum\limits_{k=0}^{n_T-1}\Bigl(\Bigl(3-\displaystyle\frac{1}{2}\varepsilon_1\Bigr)
 				\alpha^k-\Bigl(\varepsilon_2+1\Bigr)\beta^k\Bigr)
 				\|\partial_x(\omega\hat{e}_N^{k+\frac{1}{2}})\|^2\\
 				&\hspace{1cm}
 				+2\Delta t\sum\limits_{k=0}^{n_T-1}\beta^{k}\Bigl(-\|\hat{e}_N^{k+\frac{1}{2}}\|^2_{\omega^{-3,0}}
 				+\|\partial_x\hat{e}_N^{k+\frac{1}{2}}\|^2_{\omega}\Bigl)\\
 				&\leq \|\hat{e}_N^0\|^2_{\omega}
				+\Delta t\sum\limits_{k=0}^{n_T-1}
				\displaystyle\frac{\beta^k }{\varepsilon_2}\|\partial_x\tilde{e}_N^{k+\frac{1}{2}}\|^2
				+\beta^k\|\partial_x e_N^{k+\frac{1}{2}}\|^2\\
				&\hspace{1cm}+\Delta t\sum\limits_{k=0}^{n_T-1}
				C_{\alpha^k,\varepsilon_1}\Bigl(\|\partial_t{u}^{k+\frac{1}{2}}-\partial_t{u}^k\|^2_{H^{-1}(I)} 
 				+N^{-2r}\|\partial_t{u}^k\|^2_{H^{\max\{2, r-1\}}(I)}\Bigr),	
			\end{array}
 		\end{equation*}
		for every $\varepsilon_1,\varepsilon_2>0$. 
		
		Now, using the hypothesis H\ref{h.th.error} and from the fact that
		\[\|\partial_x e_N^{k+\frac{1}{2}}\|^2\leq 2\|(1-x)\partial_x (\omega\hat{e}_N^{k+\frac{1}{2}})\|^2
			+2\|\partial_x(\omega e_N^{k+\frac{1}{2}})\|^2\leq 8\|\partial_x(\omega\hat{e}_N^{k+\frac{1}{2}})\|^2,\] 
		the above inequality can be transformed  by:
		\begin{equation}\label{x.aux.proof}
 			\begin{array}{lll}
 				&\|\hat{e}_N^n\|_{\omega}^2
 				+\Delta t\sum\limits_{k=0}^{n_T-1}\Bigl(\Bigl(\displaystyle\frac{3}{8}-\displaystyle\frac{1}
 				{16}\varepsilon_1\Bigr)\alpha^k
 				-\Bigl(\displaystyle\frac{1}{8}\varepsilon_2+\displaystyle\frac{9}{8}\Bigr)\beta^k\Bigr)
 				\|\partial_x{e}_N^{k+\frac{1}{2}}\|^2\\
 				&\hspace{1cm}
 				+2\Delta t\sum\limits_{k=0}^{n_T-1}\beta^{k}\Bigl(-\|\hat{e}_N^{k+\frac{1}{2}}\|^2_{\omega^{-3,0}}
 				+\|\partial_x\hat{e}_N^{k+\frac{1}{2}}\|^2_{\omega}\Bigl)\\
 				&\leq \|\hat{e}_N^0\|^2_{\omega}
				+\varepsilon_2^{-1}\|\beta\|_{L^\infty([0,T])}\Delta t\sum\limits_{k=0}^{n_T-1}
				\displaystyle\|\partial_x\tilde{e}_N^{k+\frac{1}{2}}\|^2\\
				&\hspace{1cm}
				+\displaystyle C\varepsilon_1^{-1}\Bigl(\min\limits_{t\in [0,T]}\alpha(t)\Bigr)^{-1}
				\Delta t\sum\limits_{k=0}^{n_T-1}
				\Bigl(\|\partial_t{u}^{k+\frac{1}{2}}-\partial_t{u}^k\|^2_{H^{-1}(I)} 
 				+N^{-2r}\|\partial_t{u}^k\|^2_{H^{\max\{2, r-1\}}(I)}\Bigr),
			\end{array}
 		\end{equation}
		for every $\varepsilon_1,\varepsilon_2>0$. 
		
		Finally, we estimate the terms in the right--hand side of 
		\eqref{x.aux.proof}. To do that, we use Lemma \ref{lema.properties.pi} and \cite{2000Li}. Then,
		a direct computation allows us to obtain the inequalities:
		\begin{equation*}
		 	\begin{array}{lll}
		 	&\Delta t\sum\limits_{k=0}^{n-1}\|\partial_t{u}^{k+\frac{1}{2}}-\partial_t{u}^k\|^2_{H^{-1}(I)}
			\leq C(\Delta t)^4 \|\partial_t^3{u}\|_{L^2(0,T;H^{-1}(I))}^2,\\
			&\Delta t\sum\limits_{k=0}^{n-1}\|\partial_t{u}^k\|^2_{H^{\max\{2, r-1\}}(I)}
			\leq C\|\partial_t{u}\|^2_{L^2(0,T;H^{{\max\{2, r-1\}}}(I))},\\
			\sum\limits_{k=0}^{n_T-1}
			\displaystyle\|\partial_x\tilde{e}_N^{k+\frac{1}{2}}\|^2
			& \leq CN^{2(1-r)}\|\partial_x^r u\|^2_{\omega^{r-2,r-2}},\quad \forall r\geq 2\quad  
			\mbox{and}\quad
			\|\tilde{e}^0\|_{\omega}\leq CN^{-r}\|u_0\|_{H^r(I)}.	
		 	\end{array}
		\end{equation*}
		Therefore, using the above estimates and the triangular inequality in \eqref{x.aux.proof}, the
		desired inequality is obtained. This arguments complete  the proof of Theorem \ref{teo.error}.
	\end{proof}
	\begin{remark}
		It is worth mentioning that the convergence analysis only for the KdV equation 
		with the same method was done in \cite{2000Maetal} for constant coefficients, that means, 
		in \eqref{sys.kdvb.linear} 
		$\alpha(t)=1$ and $\beta(t)=0$, for $t\in (0,T)$. Obviously, estimate \eqref{ine.teo.error1} implies 
		the case proven in \cite{2000Maetal}. On the other hand, by using dual--Petrov--Galerkin bases,
		\cite[Theorem 2.2] {2003ShenJie} involves a second--order term with constant coefficient 
		$\beta$ in the range $(-\frac{1}{3},\frac{1}{6})$, and $\alpha=1$. Note that our hypothesis
		H\ref{h.th.error} satisfies such parameter configuration when, for instance,  
		$\varepsilon_1,\varepsilon_2$ go to zero. 
	\end{remark}

% -----------------------------------------
% -------IMPLEMENTATION----------------------------------------
%-----------------------------------------	
\subsection{Implementation scheme.}\label{subsection.scheme} In this paragraph we discuss the numerically implementation of  the fully discrete spectral method given in \eqref{sys.full.discrete}. We need to solve at each time level the problem of finding $u_N^k\in V_N$ verifying 
	\begin{equation}\label{eq.weakform}
		(u_N^{k+1},v)+\frac{\Delta t}{2}\alpha(k\Delta t)(\partial_x^3u_N^{k+1},v)
		-\frac{\Delta t}{2}\beta(k\Delta t)(\partial_x^2u_N^{k+1},v)=(g^k,v)\quad \forall v\in W_{N-1},
	\end{equation}
	where 
	\[g^k=u_N^k-\frac{\Delta t}{2}\alpha(k\Delta t)\partial_x^3u_N^{k}
	+\frac{\Delta t}{2}\beta(k\Delta t)\partial_x^2u_N^{k}+\Delta tf^{k+\frac{1}{2}}.\] 
	
	Therefore, by considering $u_N^k(x)=(1-x)\sum\limits_{n=0}^{N-3}\hat{u}_n^k\phi_n(x)$
	and taking as test function $v=\phi_m$, for $m=\overline{[0,N-3]}$, the above identity can be written by
	\begin{equation*}
	\begin{array}{lll}
		&\sum\limits_{n=0}^{N-3}\Biggl( ((1-x)\phi_n,\phi_m)+\Delta t\alpha(k\Delta t)(L_{n+1},L_{m+1})
		-\displaystyle\frac{\Delta t}{2}\alpha(k\Delta t)((1-x)L_{m+1},\partial_x L_{n+1})\Biggr)\hat{u}_n^{k+1}\\
		&\hspace{1cm} +\displaystyle\frac{\Delta t}{2}\beta(k\Delta t)\sum\limits_{n=0}^{N-3}\Biggl( ((1-x)L_{n+1},L_{m+1})
		+(L_n-L_{n+2},L_{m+1})\Biggr)\hat{u}_n^{k+1}\\
		&=\sum\limits_{n=0}^{N-3} (RHS)_n\hat{u}_n^{k}
		+ \Delta t\sum\limits_{n=0}^{N-3} ((1-x)\phi_n,\phi_m)\hat{f}_n^{k+\frac{1}{2}},	
	\end{array}
	\end{equation*}
	where $(RHS)_n$ corresponds to
	\begin{equation*}
	\begin{array}{lll}
		 (RHS)_n=&((1-x)\phi_n,\phi_m)-\Delta t\alpha(k\Delta t)(L_{n+1},L_{m+1})
		+\displaystyle\frac{\Delta t}{2}\alpha(k\Delta t)((1-x)L_{m+1},\partial_x L_{n+1})\vspace{0.2cm}\\
		&-\displaystyle\frac{\Delta t}{2}\beta(k\Delta t)((1-x)L_{n+1},L_{m+1})
		-\displaystyle\frac{\Delta t}{2}\beta(k\Delta t)(L_n-L_{n+2},L_{m+1}).
	\end{array}
	\end{equation*}
	Based on the above representation, we build the matrices $K,L,M$ and $Q$ of size $(N-2)\times(N-2)$ with the 
	coefficients $k_{mn}, \ell_{mn}, a_{mn}$ and $q_{mn}$ defined as follows:
	\begin{equation*}
		k_{mn}=((1-x)L_{n+1},L_{m+1})+(L_n-L_{n+2},L_{m+1})=\left\{
		\begin{array}{lll}
			2c_{m+1}& &m=n,\\
			2c_m-2(m+1)c_mc_{m+1} & & n=m+1,\\
			-2(1+(m+2)c_{m+1})c_{m+2}&& n=m-1,\\
			0 && n\leq m-2,\\
			0 && n\geq m+2.
		\end{array}\right.
	\end{equation*}
	\begin{equation*}
		a_{mn}=((1-x)\phi_n,\phi_m)=\left\{
		\begin{array}{lll}
			2c_{m+1}^2(c_m+c_{m+2})& &m=n,\\
			-2c_{m+1}c_{m+2}^2(c_m+(m+3)c_{m+3}) & & n=m+1,\\
			-2c_{m+1}c_{m+2}c_{m+3}&& n=m+2,\\
			2(m+3)c_{m+1}c_{m+2}c_{m+3}c_{m+4}&& n=m+3,\\
			0 &&  n\geq m+4.
		\end{array}\right.
	\end{equation*}
	\begin{equation*}
		q_{mn}=((1-x)L_{m+1},\partial_x L_{n+1})=\left\{
		\begin{array}{lll}
			c_{m+1}-1& &m=n,\\
			2(-1)^{m+n+1} & & n\geq m+1,\\
			0 &&  n\leq m-1
		\end{array}\right.
	\end{equation*}
	and 
	\begin{equation*}
		\ell_{mn}=(L_{n+1},L_{m+1})=diag(2C_{m+1}).
	\end{equation*}
	Therefore, the matrix representation of problem \eqref{sys.full.discrete} is 
	\begin{equation}\label{matrixrepre}
		\mathbf{A}\mathbf{U}^{k+1}=\mathbf{B}\mathbf{U}^{k}+\mathbf{C}\mathbf{F}^{k+\frac{1}{2}},
	\end{equation}	
	with 
	\begin{equation}\label{matrixA}
		\mathbf{A}=M+\Delta t\alpha(k\Delta t) L-\frac{\Delta t}{2}\alpha(k\Delta t) Q+\frac{\Delta t}{2}\beta(k\Delta t) K,
	\end{equation}
	and 
	\begin{equation}\label{matrixB}
		\mathbf{B}=M-\Delta t\alpha(k\Delta t) L+\frac{\Delta t}{2}\alpha(k\Delta t) Q-\frac{\Delta t}{2}\beta(k\Delta t) K,
	\quad \mathbf{C}=\Delta tM.
	\end{equation}

	\begin{remark}
		Since our approach is based in \cite{2000Maetal}, the above mass matrix $\{a_{mn}\}$ is the same, as well as
		a part of the term $g^k$ given in \eqref{eq.weakform}. However, in \cite{2000Maetal} there are some 
		typos and incomplete information related to implementation that we have corrected in the present work.  	
	\end{remark}
	In order to display the  stability property for the fully discrete scheme \eqref{eq.weakform},
	in Figure  \ref{fig.eigenvalues} we plot the eigenvalues for the LPG discretization for two parametric configurations of 
	$\alpha$ and $\beta$. In relation to \cite{2000Maetal} where the stability analysis was done only for the 
	KdV equation with $\alpha=1$, we observe that for the KdVB equation with time dependent coefficients does not exist instable modes, neither. 
	Nevertheless, as mentioned in Remark \ref{obs.estability1}, it is possible if the constrains upon
	the dispersion--dissipation parameters established in Theorem \ref{th1.stability} are verified. Indeed, 
	the case where $\alpha$ is small enough (tends to zero) leads to instables modes for small values of $\beta$ and therefore, it is cannot
	be considered under this setting.  Additionally, Figure  \ref{fig.eigenvalues} allow us to visualize the 
	effect caused by the interaction among dispersion and dissipation parameters in the eigenvalues distribution of the spectral approximation.

\begin{figure}[!ht] 
\hspace{-.6 cm}
\includegraphics[width=8.2cm,trim=0 0 80 0,clip]{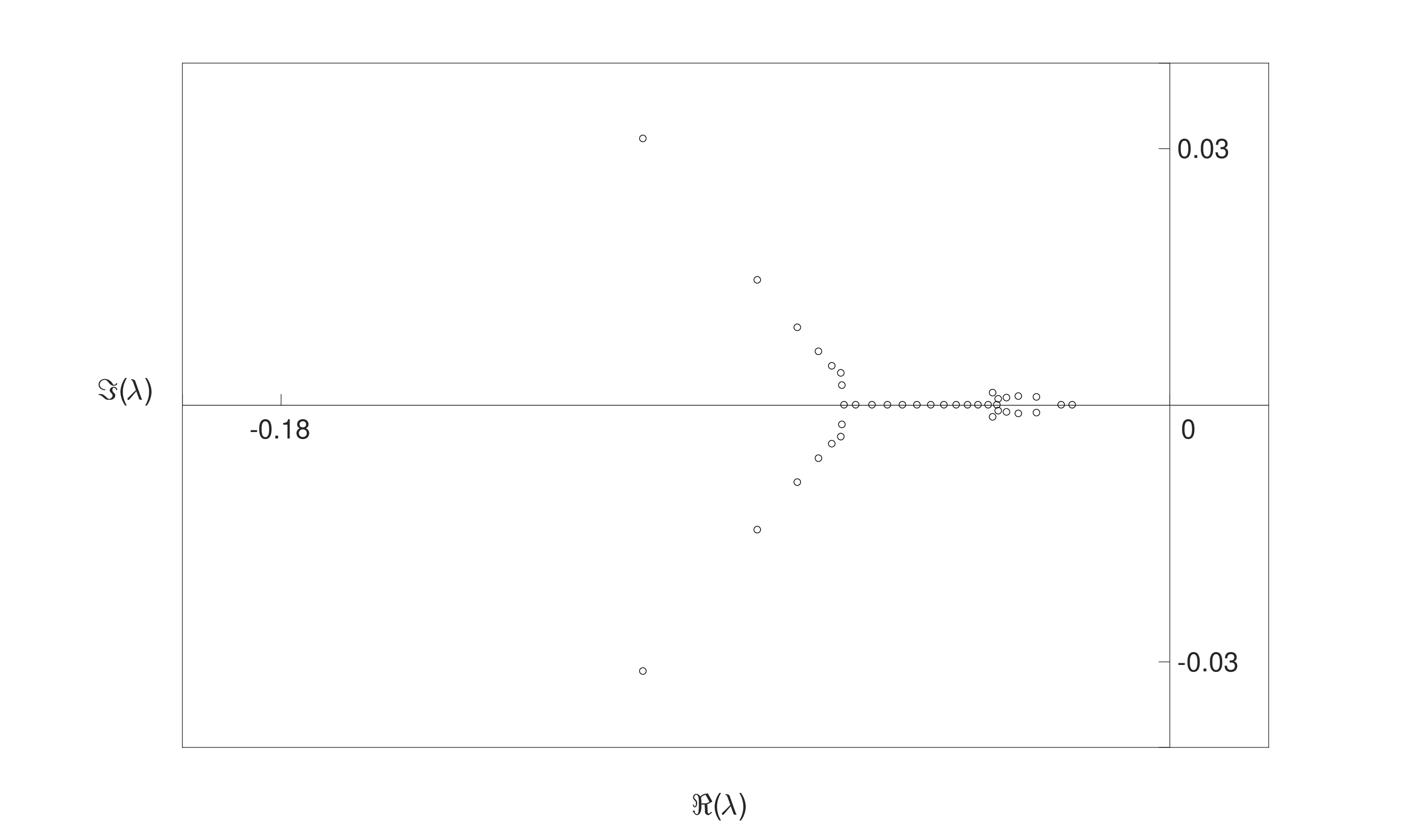}%\hspace{-0.1cm}
\includegraphics[width=8.2cm,trim=0 0 80 0,clip]{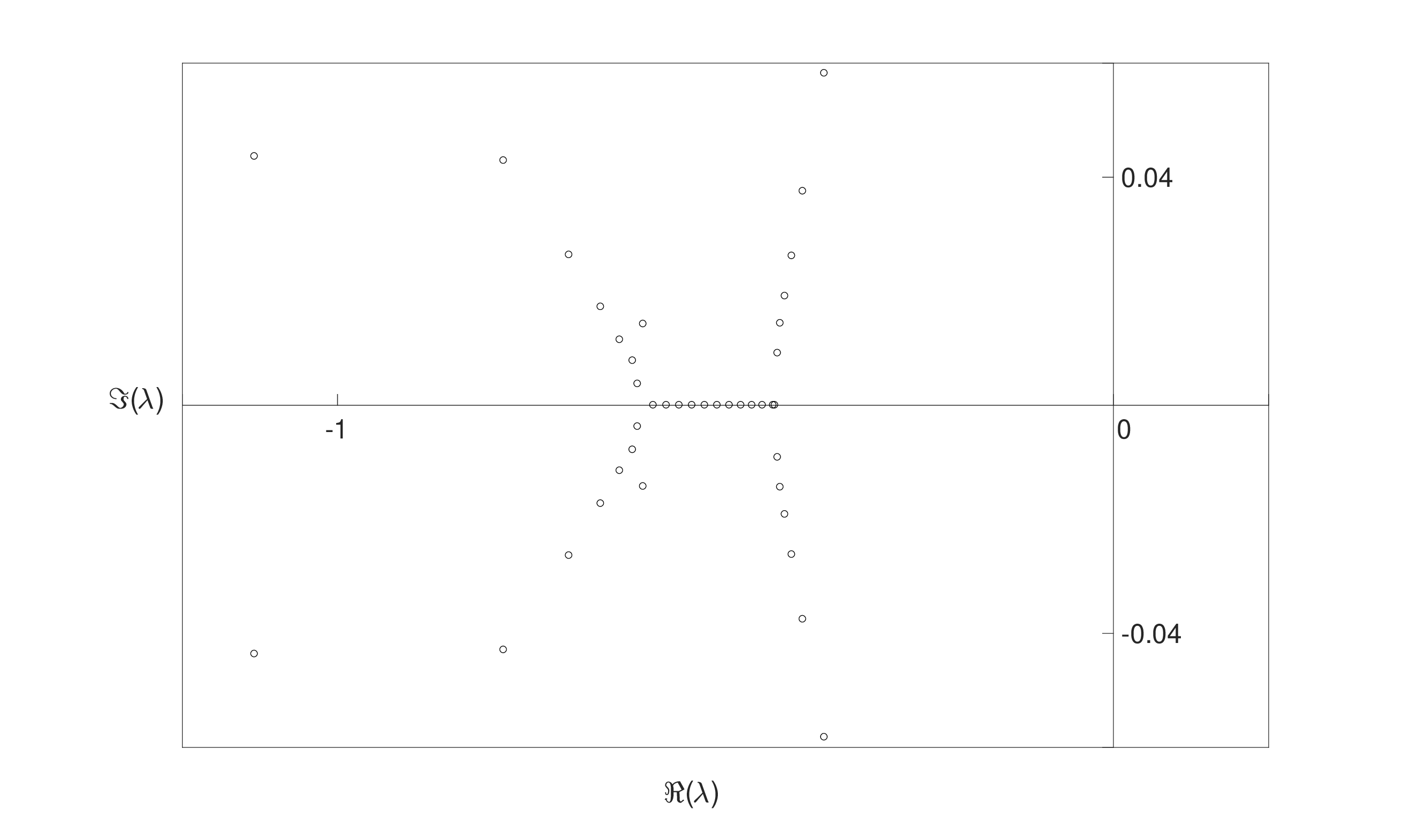}
\caption{Eigenvalues distribution of the spectral approximation for the KdVB equation, associated to the scheme 
	\eqref{eq.weakform} and matrix \eqref{matrixA}. In both cases $N=42,\, \Delta t=1$, and parameters 
	(Left) $\alpha=0.1$,\,$\beta=0.1$, (Right) $\alpha=1$,\,$\beta=0.3$.}
   \label{fig.eigenvalues}
\end{figure}

% -----------------------------------------
% -------NUMERICAL RESULTS----------------------------------------
%-----------------------------------------
\section{Numerical Results}\label{section.numerica}

	In this section, we present results obtained from simulations of the LPG method for the KdVB equation
	with dispersion--diffusion variable coefficients, see  \eqref{sys.kdvb.linear}. Recall that aspects related to stability and convergence 
	have been previously studied in Theorem \ref{th1.stability} and Theorem \ref{teo.error}. In fact, those results depend on, 
	at least, four parameters, namely, $N$, $\Delta t$, $\alpha$ and $\beta$. In order to seek the behaviors associated to the 
	theoretical descriptions and to separate the effects of each corresponding parameter, it is necessary to split the numerical problem
	in several experiments with different interactions of the aforementioned parameters. Indeed, we provide a proper calibration
	for the dispersion and diffusion parameters, which in turn shows numerical evidences to particular cases presented in previous works.
% -----------------------------------------
% -------EXPERIMENTAL SETUP----------------------------------------
%-----------------------------------------	
\subsection{Experimental setup}
	In all the cases tested, we try to set a benchmark that allow to measure every numerical experiment in a unique from. 
	To be more exact, by considering that several variables are involved into the analysis, makes it necessary to introduce 
	a fair measure which would be able to give precise information of the accuracy of the experiment. It is with this aim that we 
	define the following functions, which satisfy  \eqref{sys.kdvb.linear}.
	Henceforth, the initial distribution is defined by
	\begin{equation}\label{CI}
		u(x,0)=\sin^2(a x)\sin(bx)%(sin(a*x))^2*sin(b*x+c*(t-1)*Dt);
	\end{equation}
	and the source term as follows:
	\begin{equation}\label{CI2}
	\begin{array}{lll}
		f(x,t)&=\left((c-b^3\alpha(t))\sin^2(ax)+6a^2b\alpha(t)\cos(2ax)\right)\cos(bx+ct)\\
		&-a\alpha(t)(4a^2+3b^2)\sin(2ax)\sin(bx+ct)\\
		&+\beta(t)\left(-(2a^2\cos(2ax)-b^2\sin^2(ax))\sin(bx+ct)-2ab\sin(2ax)\cos(bx+ct)\right),	
	\end{array}
	\end{equation}
	where $a=\pi$ and $b=c=12$. 
	
	Taking into account the above data, the unique solution to \eqref{sys.kdvb.linear} can be obtained analytically. In fact, 
	the explicit form of the solution is given by
	\begin{equation}\label{SCI}
		u(x,t)=\sin^2(a x)\sin(bx+ct).
	\end{equation}
	Note that  the source (\ref{CI2}) really corresponds to a biparametric family, although its associated solution \eqref{SCI} 
	is free of parameters. In other words, all information concerning to the dispersion--dissipation 
	parameters is located in the source instead of the solution, which allow us to create uniformly measurement errors 
	respect to the analytical solution. Therefore,  we define
	the error $\epsilon$ in the norms $L^p$ ($1\leq p< \infty$) in space and $L^1$ in time as follows:
	\begin{equation}\label{error}
		\epsilon:=\frac{\Delta t}{N}\displaystyle\sum_{k=0}^{n_T}\left(\sum_{j=0}^N\left|u(x_j,t_k)-u_N^k(x_j)\right|^p\right)^{1/p}.
	\end{equation} 
	On the other hand, it is worth pointing out that the values of $a,b$ and $c$ have been taken from \cite{2000Maetal},  where 
	the authors defined (\ref{CI})--\eqref{SCI} for the KdV equation with dispersion coefficient $\alpha=1$. 
	
	Here,  $x_j$ represents the specific position defined by the Chebyshev--Gauss--Lobatto  points 
	\cite{bookCanuto}, whereas $t_k$ is the temporal discretization defined in the previous section. Clearly, our work 
	constitutes an extension to the linear model given in  \cite{2000Maetal}. 
	%Lastly, for illustrative purposes, we plot in Figure \ref{AnalVsNum} a snapshot of the simulation among
	%the analytical solution (\ref{SCI}) and the values $u^k_N$ by solving \eqref{matrixrepre}--\eqref{matrixB} for 
	%a particular parametric configuration.
	
% -----------------------------------------
% -------TEMPORAL AND SPATIAL DISCRETIZATION----------------------------------------
%-----------------------------------------	
\subsection{Temporal and spatial discretization}
	The main task in this paragraph consists on providing evidence according to Theorem \ref{teo.error}. Since 
	the spatial and temporal approximations show rate of convergence linked to the dispersion and diffusion parameters, 
	we pretend to observe both tendencies throughout  the same numerical experiment defined by (\ref{CI})--(\ref{SCI}), 
	but using different parametric configurations depending on the case.

\begin{figure}[h]
		\begin{center}
		\begin{tabular}{cc}
		\includegraphics[width=0.5\linewidth]{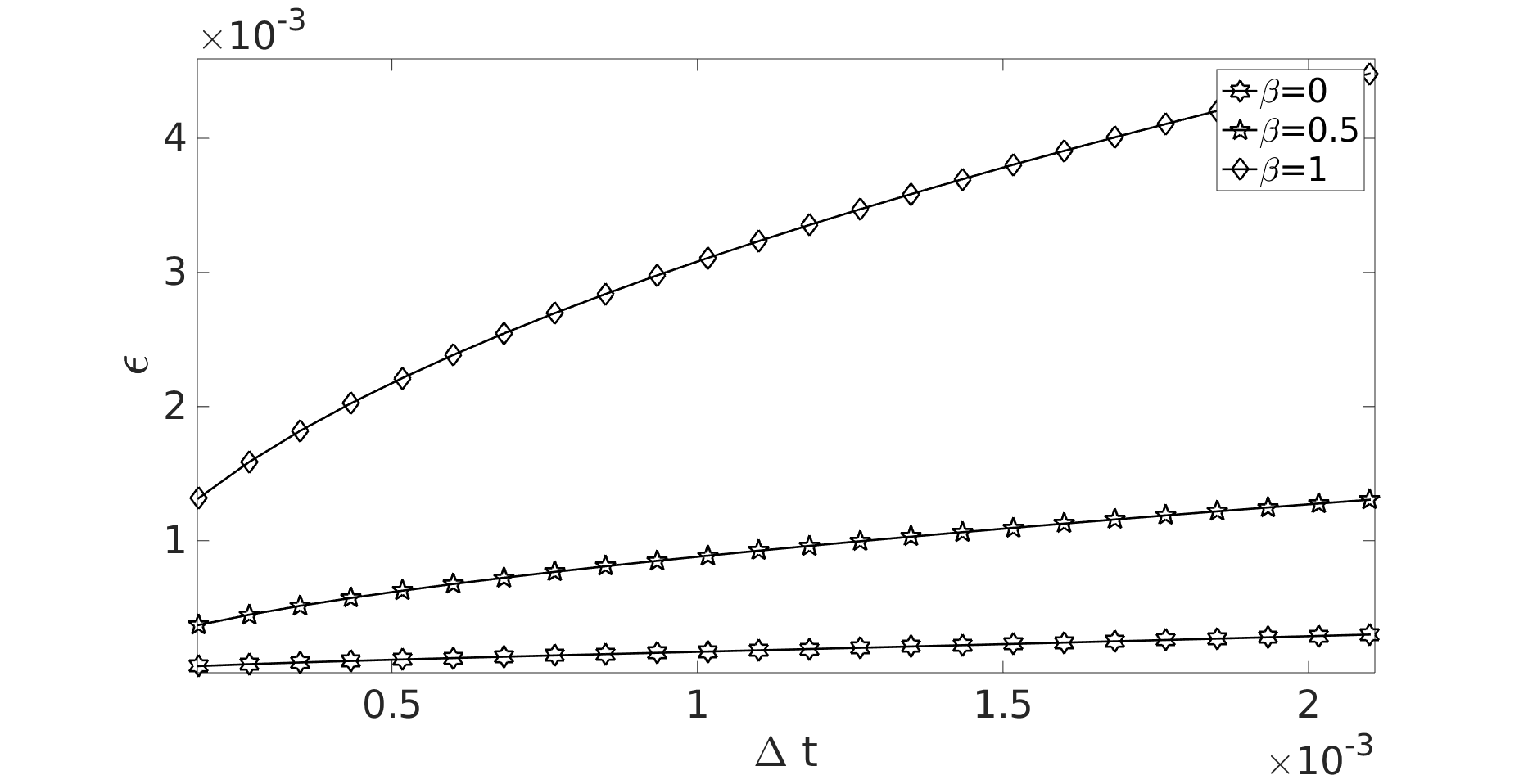} &
		\includegraphics[width=0.5\linewidth]{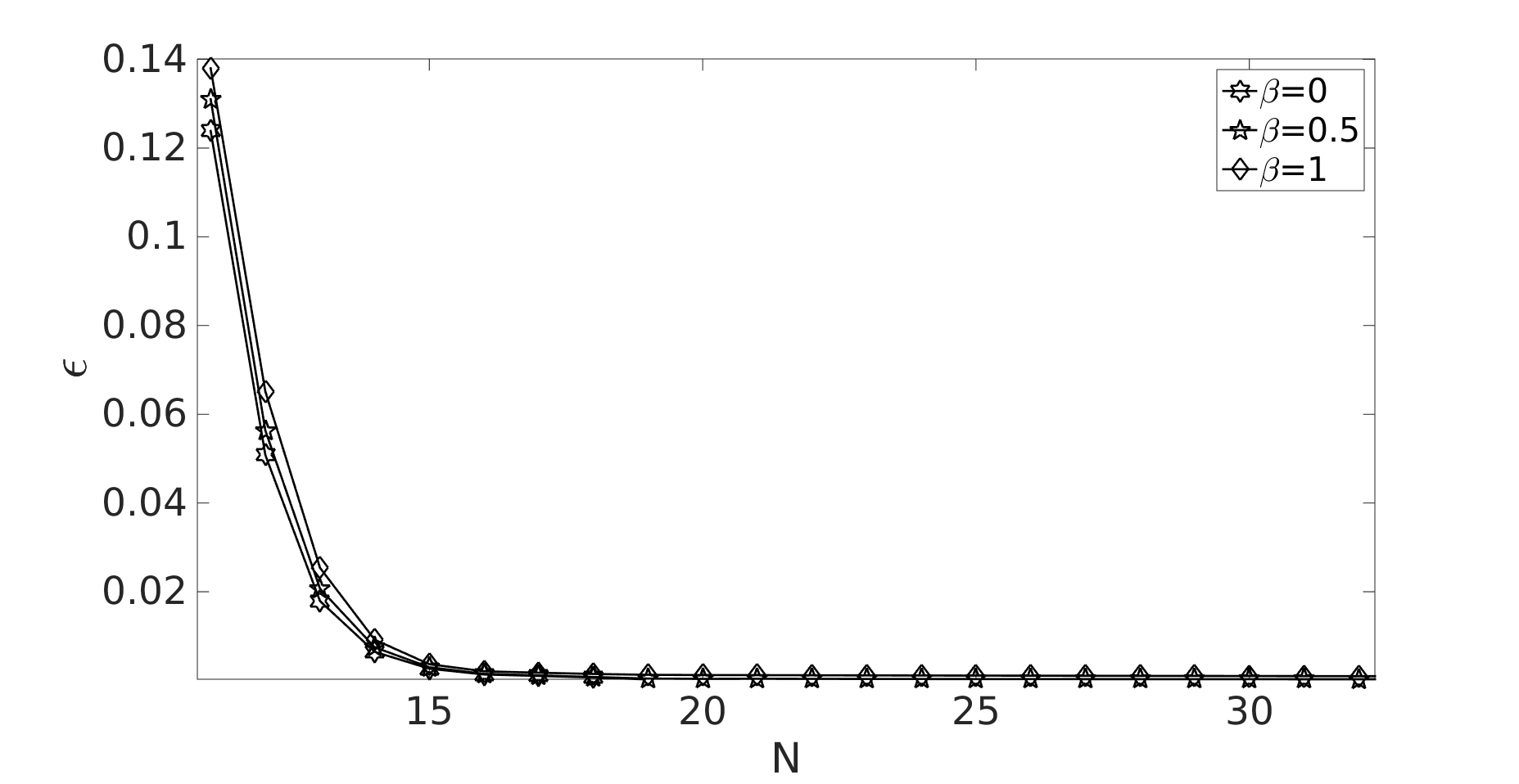} %\\(a)\\
		\end{tabular}
		\end{center}
		\caption{Each mark corresponds to the error $\epsilon$ in $L^1(L^2)$--norm for $\alpha=1$ and three different $\beta$ values.
		(Left) temporal convergence and (Right) spatial convergence.
		}
		\label{ConvDtDx}
	\end{figure}

	\begin{table*}[h]
  \begin{center}
   \caption{Analysis of the temporal convergence using least squares method, and  fitting the solution of the numerical problem
		 by 	$u^k_N=u(x_n,t_k)[1+(\Delta t)^{\text{order}}]$.
		In all simulations we fix $\alpha=1$ and $N=32$. Here, 
		the extrapolated error values  using  (\ref{error}) with $p=1$ and $p=2$ as well as their corresponding 
		convergence orders are recorded.}
   		 \begin{tabular}{|c||c|c|c|c|}
		\hline
      		$\beta$ & $ L^1(L^1)\, -\text{norm}$ & \text{order} & $ L^1( L^2)\, -\text{norm}$ & \text{order}  \\
	\hline\hline
       $0$ & 0.01335901 & 1.87 & 0.00010075 & 2.80 \\
	\hline
       $0.2$ & 0.04785512 & 1.90 & 0.00046410 & 2.72  \\
      \hline
       $0.4$ & 0.16367134 & 2.10 & 0.00169221 & 2.54 \\
	\hline
       $0.6$ & 0.39841512 & 2.48 & 0.00316313 & 2.23 \\
      \hline
       $0.8$ & 0.71474315 & 2.62 & 0.00485123 & 1.71 \\
	\hline

    \end{tabular}\label{TABLAEST}
  \end{center}
\end{table*}

	First, we focus on the temporal convergence of the method. As starting point, we develop a massive experiment fixing the 
	number of modes and the dispersion parameter, namely,  $N=32$ and  $\alpha=1$, respectively. Besides, for five 
	different $\beta$ values, e.g.   $\beta=\{0,0.2,0.4,0.6,0.8\}$, we carry out $20$ simulations of $T=2$ seconds of time by fixing 
	steps $\Delta t$ in the range $\{(i+1)10^{-4}: i=1,\dots, 20\}$. In order to achieve a better visualization and without
	less of generality, we only depict three different cases of $\beta$ values in Figure \ref{ConvDtDx} (Left), where each curve
	represents a dispersion parameter $\beta$ and every mark shows the error obtained upon its corresponding temporal step $\Delta t$. 
	Note that, by fixing the dispersion coefficient at $\alpha=1$, we clearly observe a convergence of the value of $\epsilon$  to zero, 
	although its rate of convergence is affected by the $\beta$ dissipation coefficients. Motivated by this, it is interesting to study 
	the order of convergence for each curve by using a least squares fit. Table \ref{TABLAEST}  shows the five different 
	$\beta$ dissipation coefficients where both the order of convergence  and the extrapolated error values 
	are recorded. We can see that in average the temporal convergence tends to $2$ as expected from theoretical result, see \eqref{ine.teo.error1}.

\begin{table*}[h]
  \begin{center}
    \caption{Error analysis in $L^1(L^2)$--norm for seven different values of $N$ and five different values of $\beta$. 
		 Besides,  $\alpha=1$ and $\Delta t=10^{-4}$ seconds.}
    \begin{tabular}{|c||c|c|c|c|c|c|c|}
	\hline
      $\beta$  & $N=14$ & $N=16$ & $N=18$ & $N=20$ & $N=22$ & $N=24$ & $N=26$  \\
	\hline\hline
       $0$   & 0.00645621 & 0.00135455  & 0.00066981 & 0.00011939 & 0.00015875 & 0.00010339 & 0.00008585  \\
	\hline
       $0.2$ & 0.00691887  & 0.00140829 & 0.00069669 & 0.00018076 &  0.00020508 & 0.00016183 & 0.00014666  \\
      \hline
       $0.4$   & 0.00764791 & 0.00153235 & 0.00081798 & 0.00042376 & 0.00038715 & 0.00036853 & 0.00035271  \\
	\hline
       $0.6$  & 0.00865866 & 0.00183839 & 0.00119356 & 0.00091525 & 0.00088213 & 0.00083909 & 0.00080551  \\
      \hline
       $0.8$  & 0.00983571 & 0.00240629 & 0.00185255 & 0.00161797 & 0.00155112 &  0.00148371 & 0.00142658  \\
	\hline
    \end{tabular}\label{TABLAEST2}
  \end{center}
\end{table*}

	Now,  the spatial convergence analysis is carried out upon the same numerical experiment  but, fixing both 
	the temporal sampling to $\Delta t=10^{-4}$ seconds and the dispersion parameter to $\alpha=1$.
 	Again, the $\beta$ parameter is evaluated at  $\beta=\{0,0.2,0.4,0.6,0.8\}$ whereas, 
	the number of nodes, $N$, is progressively selected in the range  $11\leq N\leq 32$. The results are depicted
	in Figure \ref{ConvDtDx} (Right), where we observe a sharp accuracy convergence  around $14\leq N\leq 22$ nodes. 
	In fact, Figure \ref{ConvDtDx} (Right) allow to deduce  that the convergence order is much greater than two, it makes imposible 
	to create any confident approximation based on  least squares fitting. To precise this feature, we show in 
	Table \ref{TABLAEST2}  the error values in the $L^1(L^2)$--norm  for five different $\beta$ values
	and seven different values of $N$. In concordance with our theoretical result on convergence, Theorem \ref{teo.error}, 
	we should note that for any value of $N$, the dispersive case ($\beta=0$) is always more accurate than the other choices 
	of $\beta$ values. In addition, it is worth mentioning that the results shown in Figure \ref{ConvDtDx} (Right) and
	in Table \ref{TABLAEST2} are not enough to establish which is exactly the convergence order, $r$, and the dependence of $N$ in the 
	accuracy of the results. Also, it is true that we observe a faster tendency in the accuracy  than for the temporal convergence, but the fact that $r$ is 
	only bounded and cannot be fitted as a simple number, it makes necessary to broad the vision and to study the  spatial convergence through 
	other considerations always in completely  agreement with the theory already exposed.
\begin{figure}[h]
	\begin{center}
		\begin{tabular}{cc}
		\includegraphics[width=0.5\linewidth]{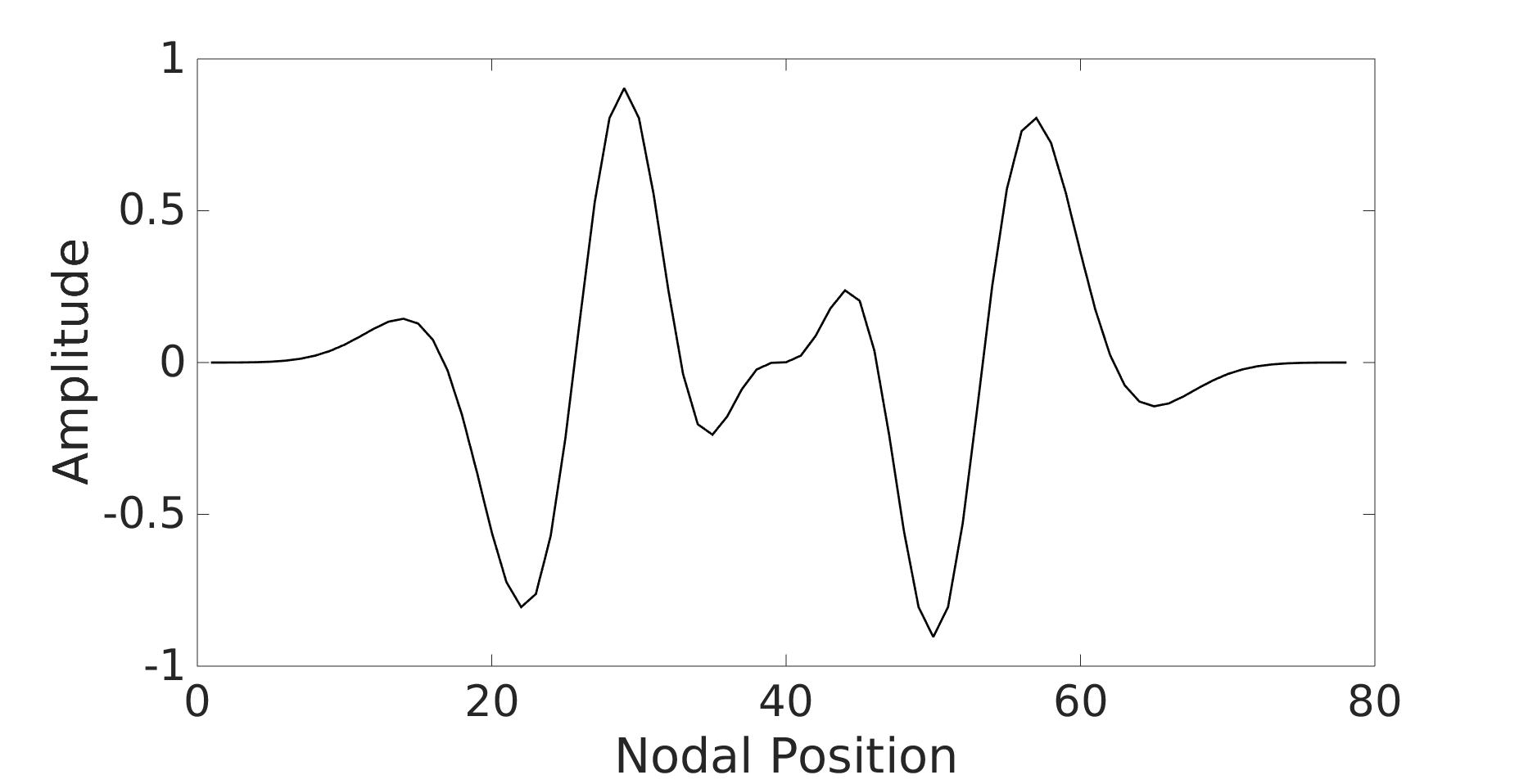}&
		\includegraphics[width=0.5\linewidth]{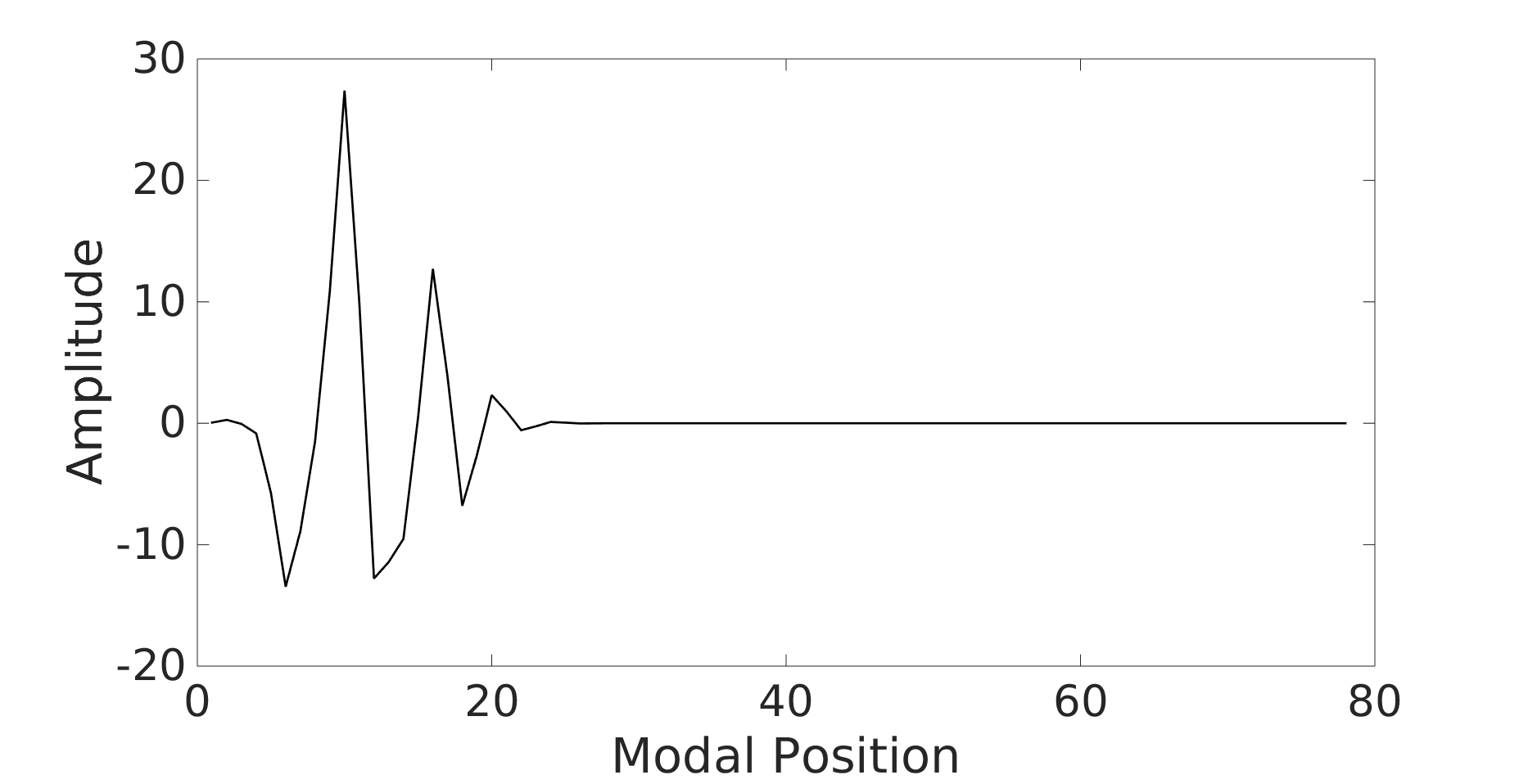} 
\end{tabular}
\end{center}
\caption{The numerical representation of the initial condition, (\ref{CI}), represented in (Left) the nodal--domain, $u^0_N$,
and (Right) the modal--domain, $\hat{u}^0_n$. In both cases we use
 $N=80$.
	 }
	\label{F_CInm}
	\end{figure}
	
\begin{figure}[h]
	\centering
		\includegraphics[scale=0.16]{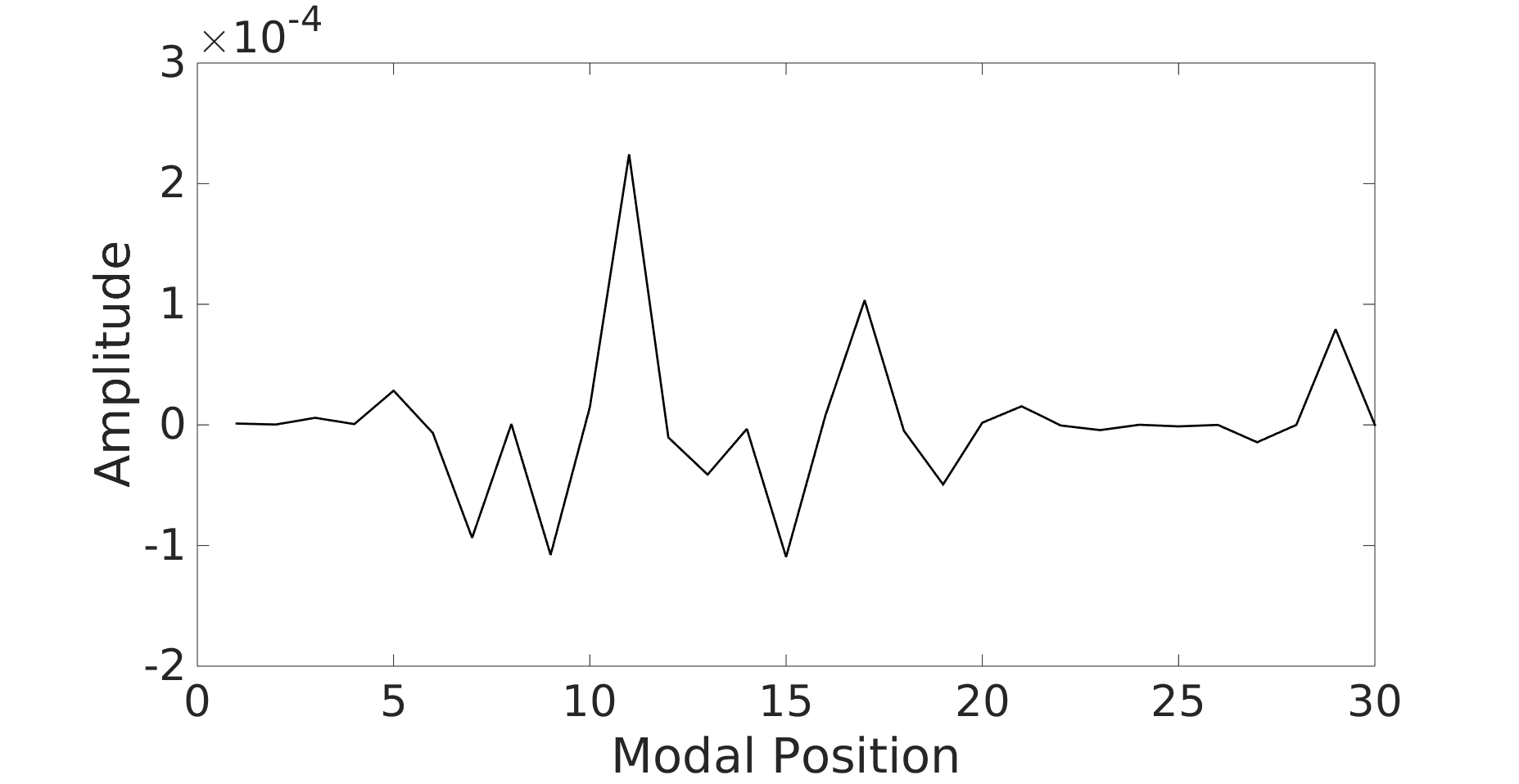}
		\includegraphics[scale=0.16]{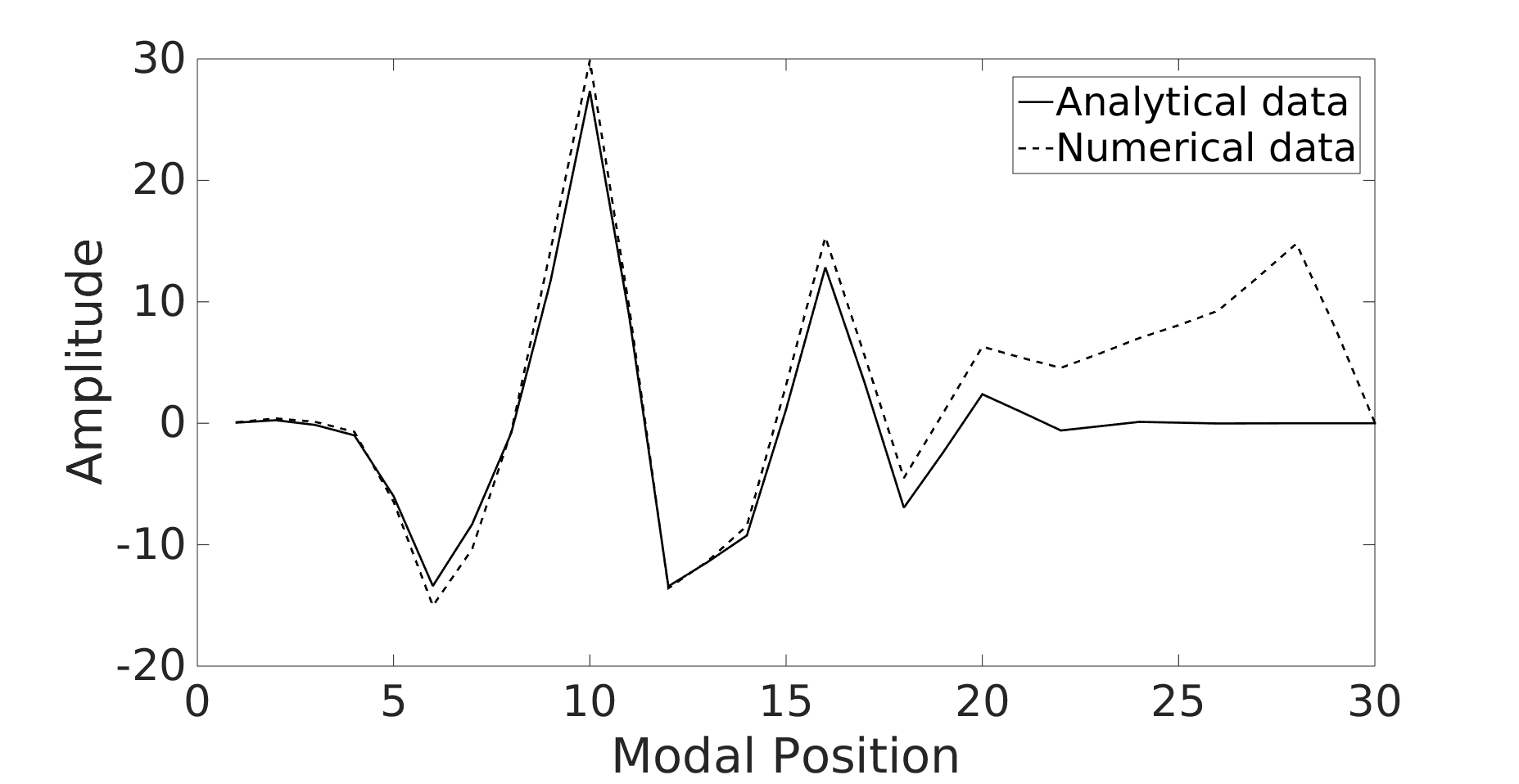} 		
		\includegraphics[scale=0.16]{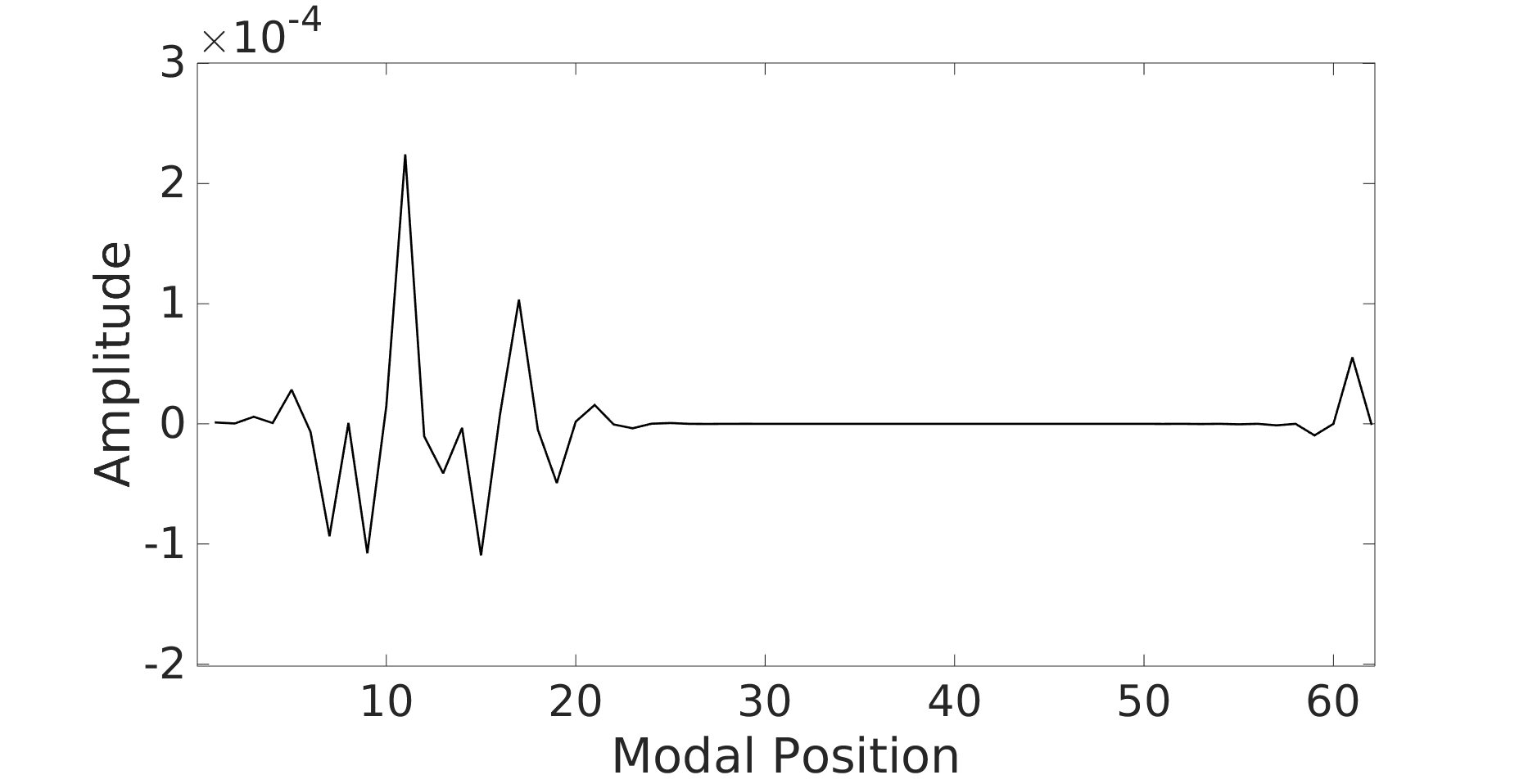} 
		\includegraphics[scale=0.16]{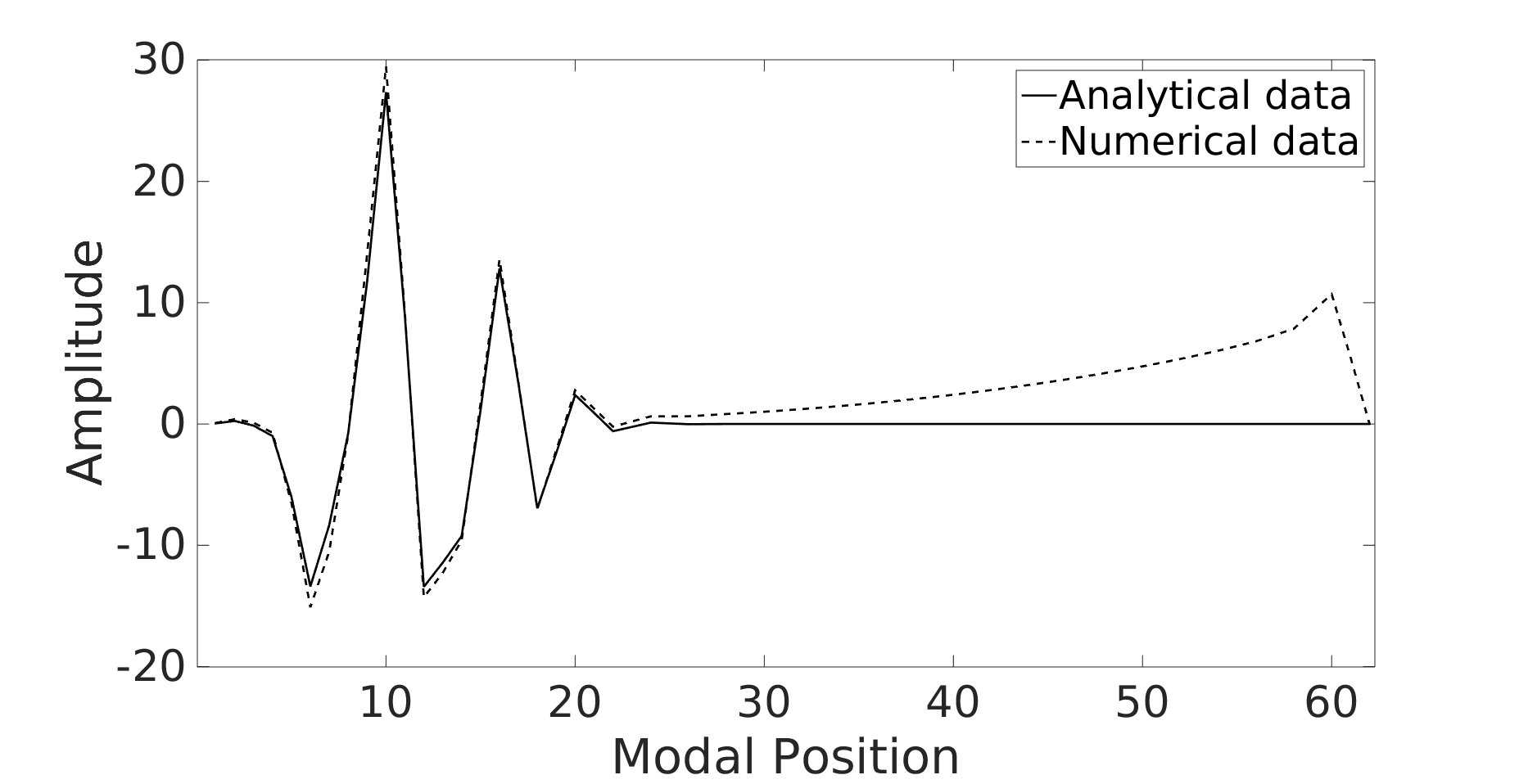} 		
		\includegraphics[scale=0.16]{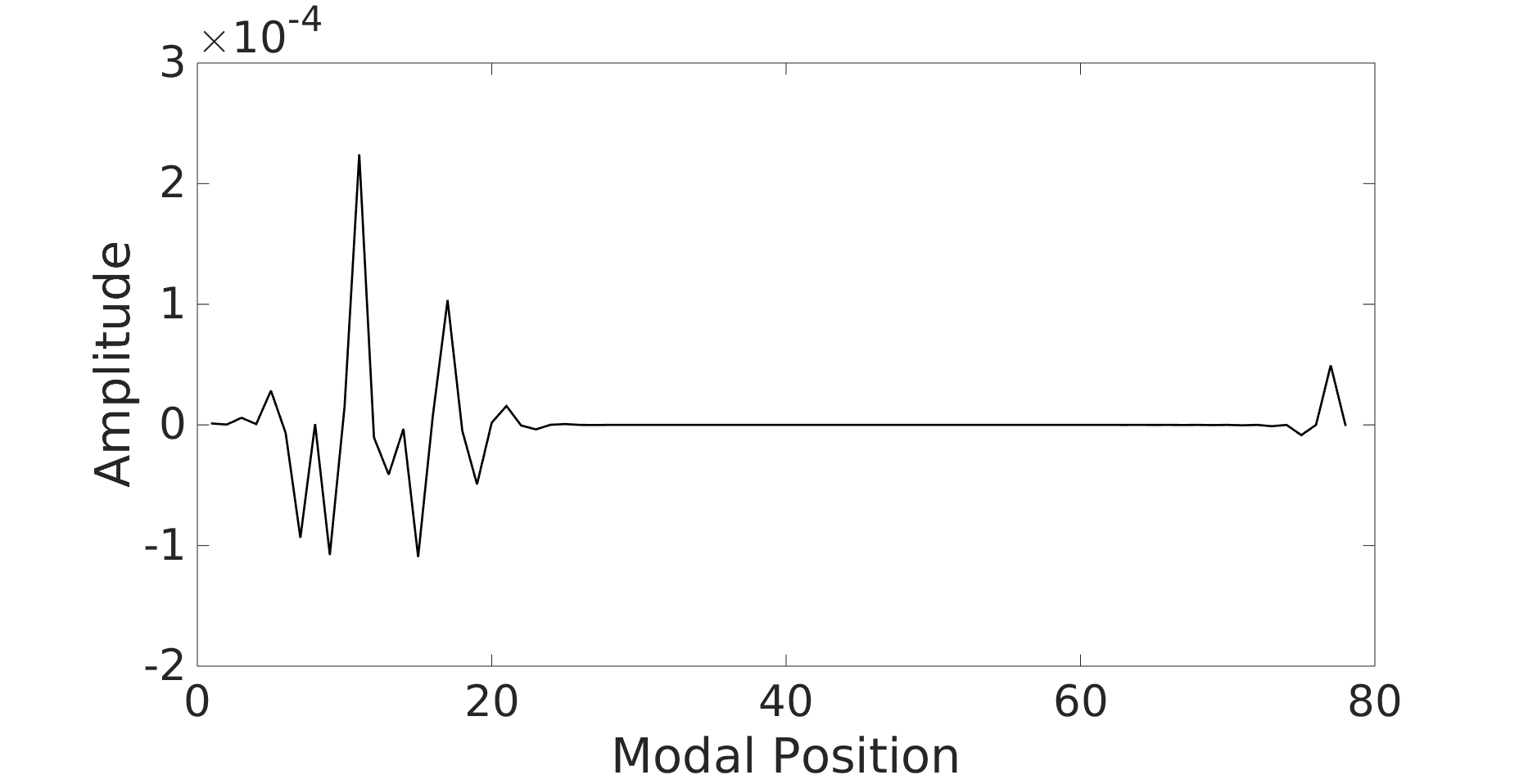} 
		\includegraphics[scale=0.16]{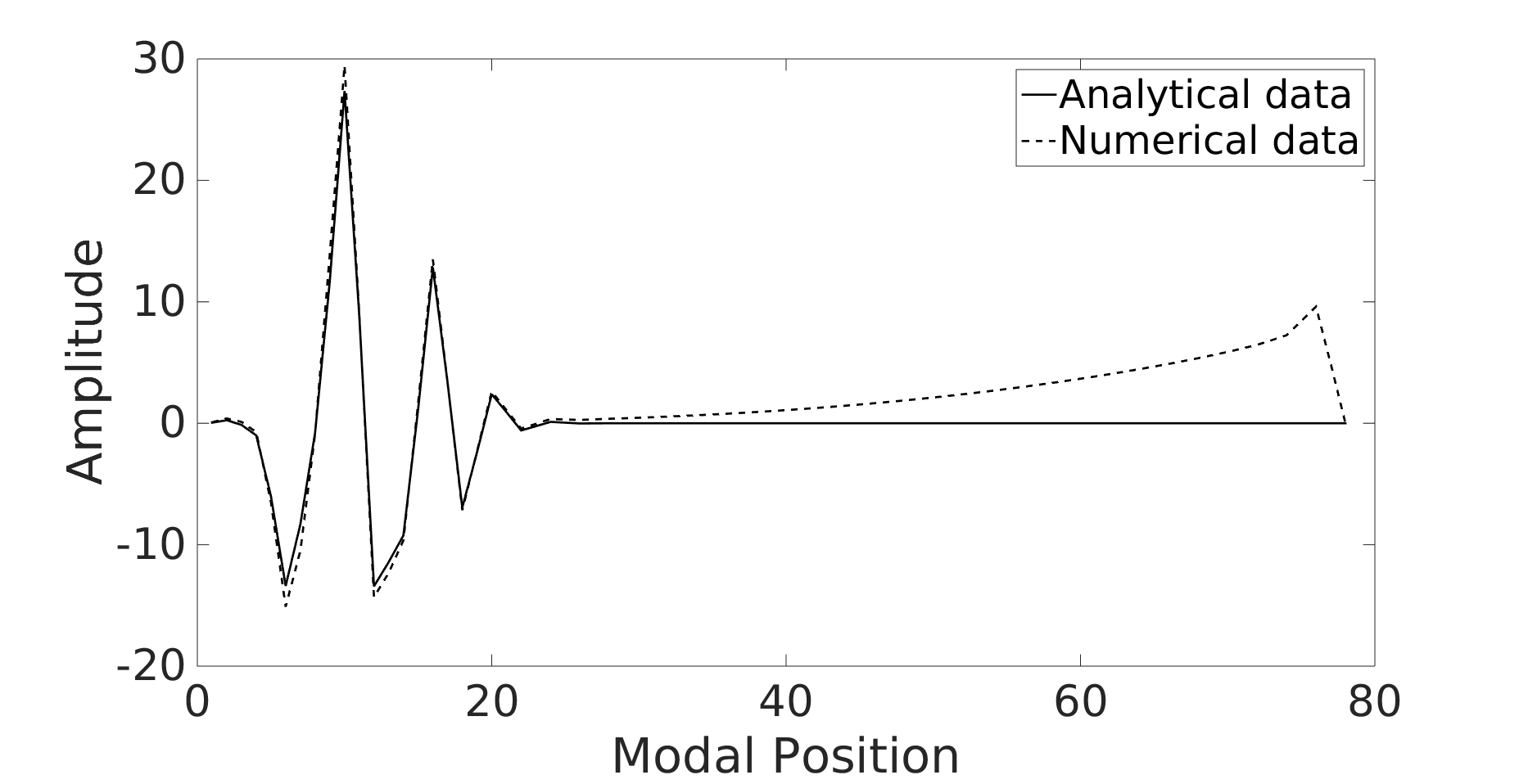} 
	\caption{In all plots, $\alpha=1$, $\beta=0.3$ and $\Delta t=10^{-5}$ seconds. Here, we check the  temporal iteration  $k=40$.
	The number of polynomial terms used are (Top row)  $N=32$, (Center row)  $N=64$ and (Bottom row)  $N=80$. 
	(Left column) the modal representation of the projected source term  $\mathbf{CF}^{k+\frac{1}{2}}$ . 
	(Right column)  the numerical solution $\hat{u}^{k+1}_n$ and the analytical solution (\ref{SCI}) in the $\phi$--space.
 }\label{Signals}
	\end{figure}
	
	Considering  the scheme developed in subsection \ref{subsection.scheme}, it is important to recall that our approach constitutes a global 
	method for the space discretization and, therefore, the computed quantities are  $\hat{u}^{k}_n$, and not the spatial values $u_N^k$.
 	Thus, the information is obtained from the modal basis ($\phi$--basis), it means that all the modes of the $\phi$--domain influence in  each 
	spatial node defined in the mesh. Therefore, it turns out interesting to observe which form takes the signal represented in 
	the modal basis  ($\phi$--basis). More precisely, we focus the analysis on both, the initial condition and the source term studying the
	main features of their modal representation. First, we take the initial condition (\ref{CI}). Note that Figure \ref{F_CInm} (Left) displays 
	the snapshots of the signal into the nodal space, $u^0_N$, whereas  Figure \ref{F_CInm} (Right) represents the transformed signal into 
	the $\phi$-domain, $\hat{u}^0_n $. The most interesting feature appears when on the right--hand side of Figure  \ref{F_CInm}, implying 
	that the main information of the signal is contained in $\phi_i\leq 22$ modes, which means that the initial condition $u_N^0$ can be constructed with  
	no more than $22$ polynomial terms $\phi_i$. Again, if we observe Figure \ref{ConvDtDx} (Right) and the values of Table \ref{TABLAEST2}, 
	it emerges a strong correlation between the spatial convergence and the modal representation of Figure \ref{F_CInm} (Right) already commented. 

	Otherwise, the source term presents more complications  making necessary to refine the analysis. In contrast to the initial condition, the source 
	term (\ref{CI2}) does not satisfy the boundary conditions stated onto the problem \eqref{sys.kdvb.linear}, since it is only defined at the  interval $(-1,1)$.  
	We highlight this fact because by construction, the $\phi$--domain transformation between the source term $f^{k+\frac{1}{2}}_N$ and the 
	modal source $\hat{f}^{k+\frac{1}{2}}_n$ only involves the homogeneous Dirichlet boundary conditions and not the Neumann conditions.
	 Indeed, contrary to the spectrum of the initial condition, for the modal signal $\hat{f}^{k+\frac{1}{2}}_n$, i.e., $\mathbf{F}^{k+\frac{1}{2}}$,  
	 the shape of their spectra always increases when $N$ is varied, suggesting that they require of infinite modes in order to be fully characterized. 
	Moreover, from (\ref{matrixrepre})--(\ref{matrixB}), we observe that the modal vector $\mathbf{F}^{k+\frac{1}{2}}$ is multiplied by the mass matrix, $M$,
	before it is introduced into the implementation scheme. Thus, the source term is transformed into a projected source term, namely, 
	$\mathbf{CF}^{k+\frac{1}{2}}$. It is worth mentioning that the transformation of $\mathbf{F}^{k+\frac{1}{2}}$ into the projected source 
	term $\mathbf{CF}^{k+\frac{1}{2}}$ provides similar spectra than Figure \ref{F_CInm} (Right) but, in that case, they also contain an artifact that 
	appears at high modes always close to $N$, independently of its value. 

	To illustrate this point, we present some results in Figure \ref{Signals}.  All the plots are obtained through simulations with 
	parameters $\alpha=1$, $\beta=0.3$ and $\Delta t=10^{-5}$ seconds. Furthermore, we represent  the data at the temporal iteration $k=40$. 
	The results represented in each row  are obtained with $N=32$, $N=64$ and $N=80$ polynomial terms, respectively. 
	Respect to the left column, this displays the results of three different modal spectra of $\mathbf{CF}^{k+\frac{1}{2}}$, whose  
	modal distribution  is concentrated in the same range of polynomial terms, $\phi_i\leq 22$ modes. However, we clearly observe an artifact 
	that always appears at the final modes of the plots, no matter the value of $N$ is employed. It is important to inform that this error has been 
	deeply analyzed by  numerical simulations obtaining the following conclusions. In addition, the case where the artifact is smaller is observed for 
	$\alpha=1,\,\beta=0$, independent of $\Delta t$,\, $N$, meanwhile,  for different values of $\beta$ considerably worsen the results. Moreover, 
	the artifact behaves  as expected in the previous theory since, for either high values of $N$ or small temporal steps, it tends to vanish. 
	On the other hand, 	the right column shows the outputs $\hat{u}^{k+1}_n$ ( that is, $\mathbf{U}^{k+1}$) and their corresponding analytical 
	values, (\ref{SCI}). In this one the analytical solutions are illustrated with solid lines, showing similar spectra than the initial data given in 
	Figure  \ref{F_CInm} (Right). In all the cases, the signals are represented with no more than $22$ polynomial terms independently of the 
	$N$ value and note that, for the rest of  the temporal iterations $k$, the solution of the problem (\ref{SCI}) may change the shape of the signal 
	depending on the $k$ value, but always preserving the limit $\phi_i\leq 22$ polynomial terms previously noticed. Concerning the signals of 
	$\mathbf{U}^{k+1}$, which are illustrated with dashed lines, these always differ from the analytical results in the high modal range presenting a 
	strong correlation with the signals of   $\mathbf{CF}^{k+\frac{1}{2}}$ depicted in Figure \ref{Signals} (Left column). 
	We observe that the artifact generated in the projected source term ($\mathbf{CF}^{k+\frac{1}{2}}$) is propagated into the numerical solution 
	$\mathbf{U}^{k+1}$ introducing numerical errors at high modes that are globally acquired in $u^{k+1}_N$.
% ------------------------------------------------------------------
% -------CALIBRATION----------------------------------------
%-------------------------------------------------------------------
\subsection{Dispersion and diffusion parameters calibration}
	Once understood the behavoir of the numerical approximation that has developed here, 
	we are able to go one step forward  and observe the influence of the dispersion and diffusion  parameters in their whole ranges. 
	As mentioned, the KdVB equation is considered to investigate the impact of bottom configurations on the free surface waves and 
	describe a wide variety of  phenomena arise in plasma physics, among others. Motivated by those applications and using as starting point the references 
	\cite{gao2015variety, 2019Liumulti, 2019Gao, jehan2011planar, 2014Hannonlinear,hussain2011korteweg}, in this subsection we develop three 
	parametric configurations among the coefficients $\alpha$ and $\beta$. Although several constant physics have been simplified in our analysis, all 
	profiles below are consistent with the previous sections and the references above mentioned. 
	
	Henceforth, all color graphics display error estimates among the analytical and numerical solution for different parameter configurations. 
	In addition, those errors are depicted in decibels (dBs)(i.e., $20\log_{10}\epsilon$), where the color white represents regions with low--error 
	values whereas dark color shows high numerical errors. In fact, errors around $-85$ dBs mean that $\epsilon\sim 5\cdot 10^{-5}$, 
	whereas errors of $-20$ dBs mean that $\epsilon\sim 10^{-1}$. Finally, all experiments  have been carried out by considering  $N=32$ nodes. 
	\begin{figure}[ht!]
		\begin{center}
		\begin{tabular}{c}
			\includegraphics[width=0.65\linewidth]{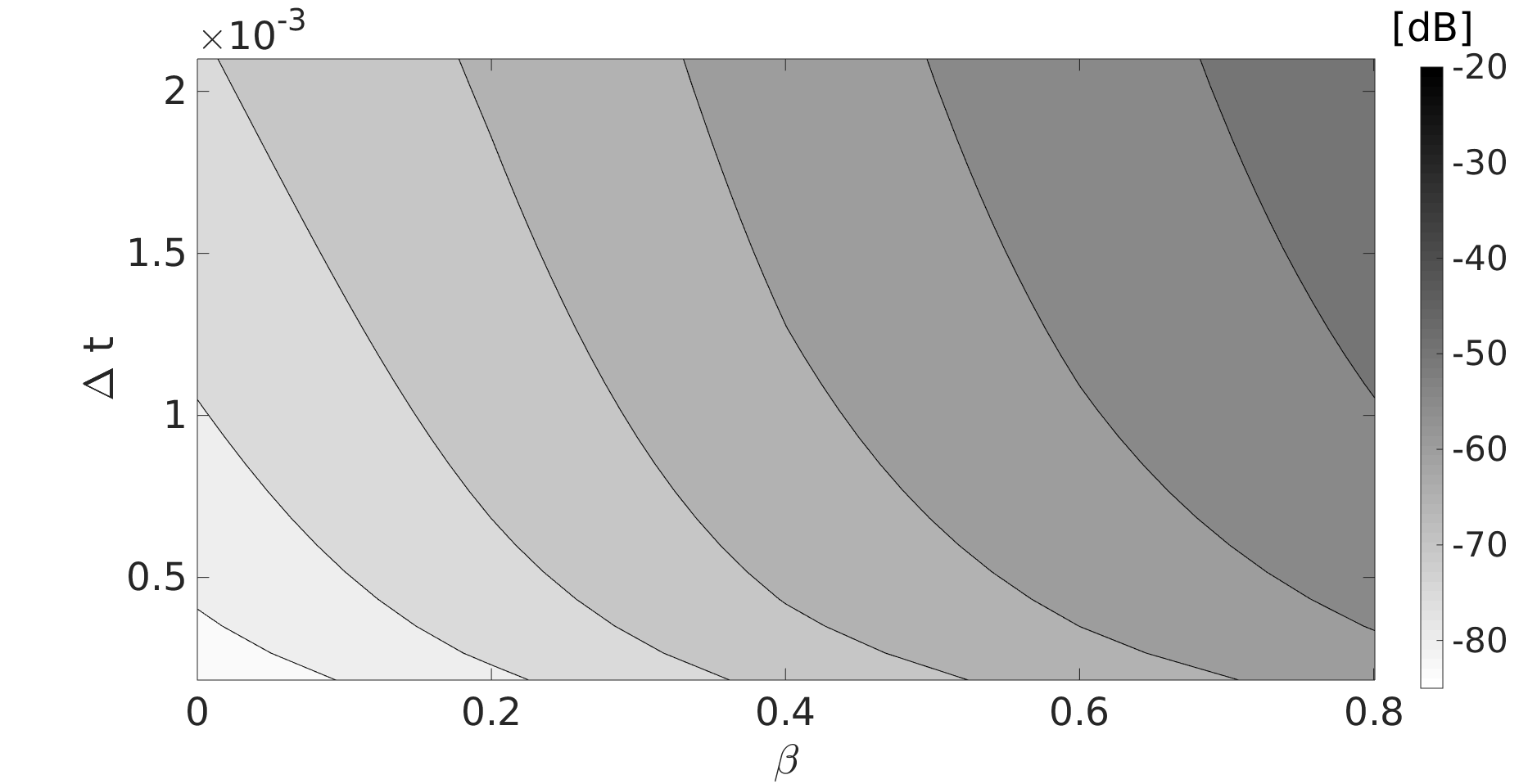} %\\(a)\\
		\end{tabular}
		\end{center}
		\caption{A color graph where the $x$-axis is the $\beta$
		parameter and the $y$-axis represents the different temporal steps $\Delta t$. Color white denotes low errors 
		meanwhile dark regions mean high errors.}
		\label{2dExperiments}
	\end{figure}

	\textit{First configuration} ($\alpha=1$ and $\beta\in [0,0.8]$).  In this case we extend the analysis associated to the temporal convergence 
	by considering  a massive experiment where the accuracy \eqref{error} is again measured for time steps $\Delta t$ in the range 
	$\{(i+1)10^{-4}: i=1,\dots, 20\}$. Moreover, the parameter $\beta$ is defined in the range $\{(i-1)4\cdot 10^{-2}: i=1,\dots, 20\}$.
	Lastly, we fix the simulation time to $T=2$ seconds. Thus, $20\times 20$  ($\beta\times\Delta t$)  simulations are carried out when the 
	dispersion coefficient is constant, namely, $\alpha=1$. Figure \ref{2dExperiments} shows the results obtained throughout the massive 
	numerical experiment described above. 
	
	In general terms, the results of the experiment seem reasonable and the spatial--temporal convergence 
	of our approach is easily observed in a wide range between $\beta$ and $\Delta t$. Note that there is a clear dependence in the errors from 
	the dissipation coefficient $\beta$. Indeed, as observed in Theorem \ref{teo.error}, the inclusion of a second order derivative in the KdV 
	equation leads to a suboptimal convergence for either the temporal or spatial discretization. Besides, with this particular configuration 
	among the dispersion and diffusion parameters, note that Figure \ref{2dExperiments} also exhibits a numerical perspective to the 
	theoretical assumption H\ref{h.th.error} given in \eqref{hypo.H.teoerror}. 
	
 	\textit{Second configuration} ($\alpha\in (0.2,1.15]$, $\beta\in [0.2,0.65)$). We present another massive numerical experiment of 
	$20\times 20$ simulations in Figure \ref{alfaVsBeta}. To be more precise, the dispersion parameter $\alpha$ belongs to the 
	set $\{0.2+(i-1)5\cdot 10^{-3}: i=1,\dots, 20\}$  and the dissipation parameter  $\beta$  belongs to the set 
	$\{(i-1)3.25\cdot 10^{-2}: i=1,\dots, 20\}$. Taking into account the above sets, we depicted two gray--scaled graphs where the color 
	again represents the error ( see (\ref{error}) with $p=2$) associated to the  specific simulation computes.
 	Figure \ref{alfaVsBeta} (Left) shows the results for $\Delta t=10^{-3}$ seconds and   Figure \ref{alfaVsBeta} (Right) for $\Delta t=10^{-4}$ seconds. 
	
	Additionally, we have selected $\alpha\in (0.2,1.15]$ and $\beta\in [0,0.65)$  because we pretend to guarantee  results with accuracy inside 
	of the range of $[-85,-20]$ dBs.  Under these intervals, we present the error values in two different cases of $\Delta t$, as mentioned
	$\Delta t=10^{-3}$  and $\Delta t=10^{-4}$ seconds. Note that, in both graphs, the shape in the color variation is preserved whereas the amount 
	of error in Figure \ref{alfaVsBeta} (Left) is constantly increased (i.e., homogeneously darker)  than for the  results of
	Figure \ref{alfaVsBeta} (Right) which errors are obtained with a smaller temporal discretization. 
	We  also highlight that,  independently of the temporal step employed, values  of $\alpha< 0.2$ provide critical errors of $\epsilon$, which are out
	 of the accuracy bounds previously established.  Moreover, we also restrict the range of $\beta$ to $0.65$ because the errors 
	become critical, i.e.,  up to $-20$ dBs, if we consider simulations with a parameter $\alpha<1/3$. To finish this case, it is worth pointing that these 
	results are in full concordance with the established relation H\ref{h.th.error}  among the dispersion and diffusion parameters.
\begin{figure}[h]
	\begin{center}
	\begin{tabular}{cc}
	\includegraphics[width=0.5\linewidth]{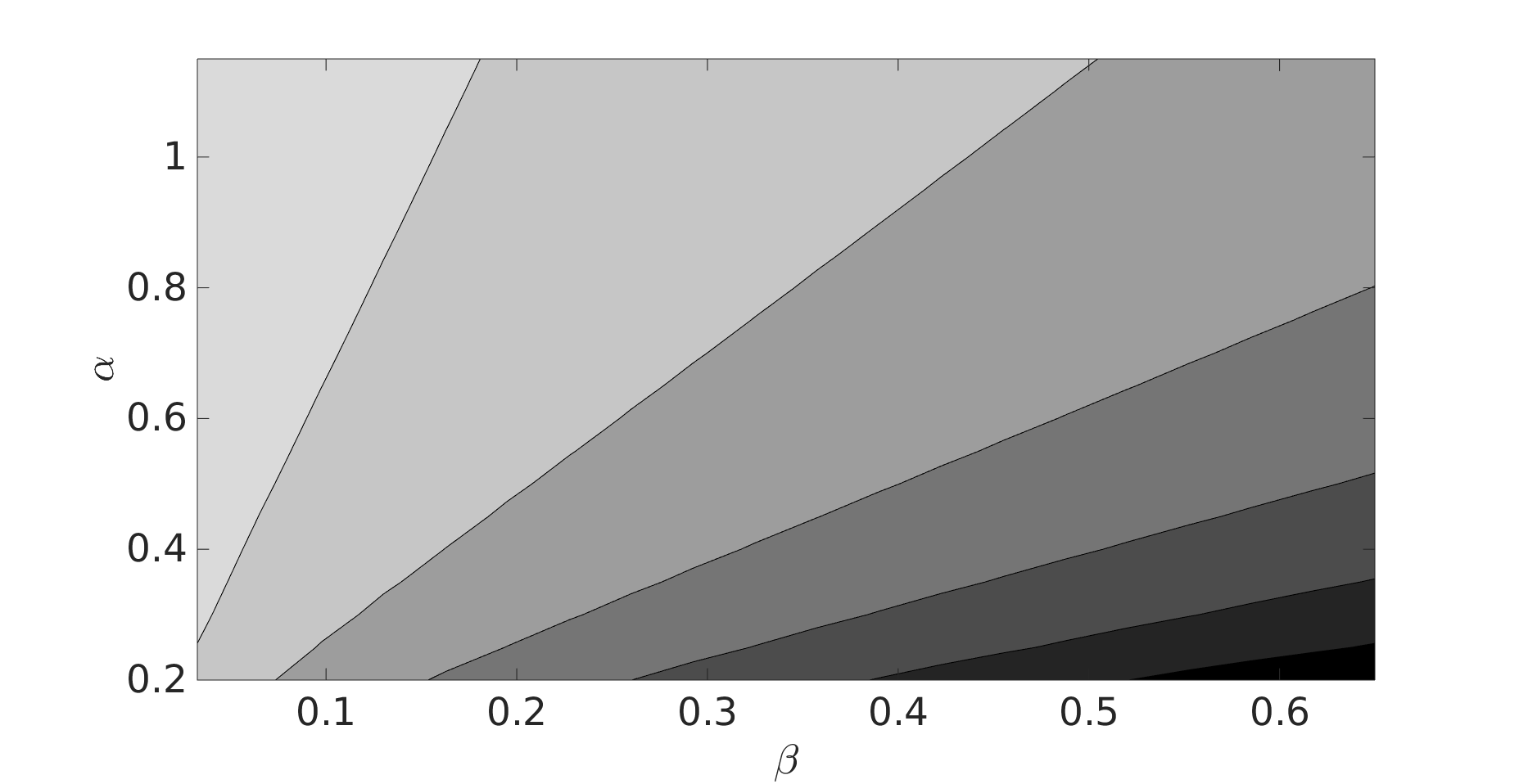} &
	\includegraphics[width=0.5\linewidth]{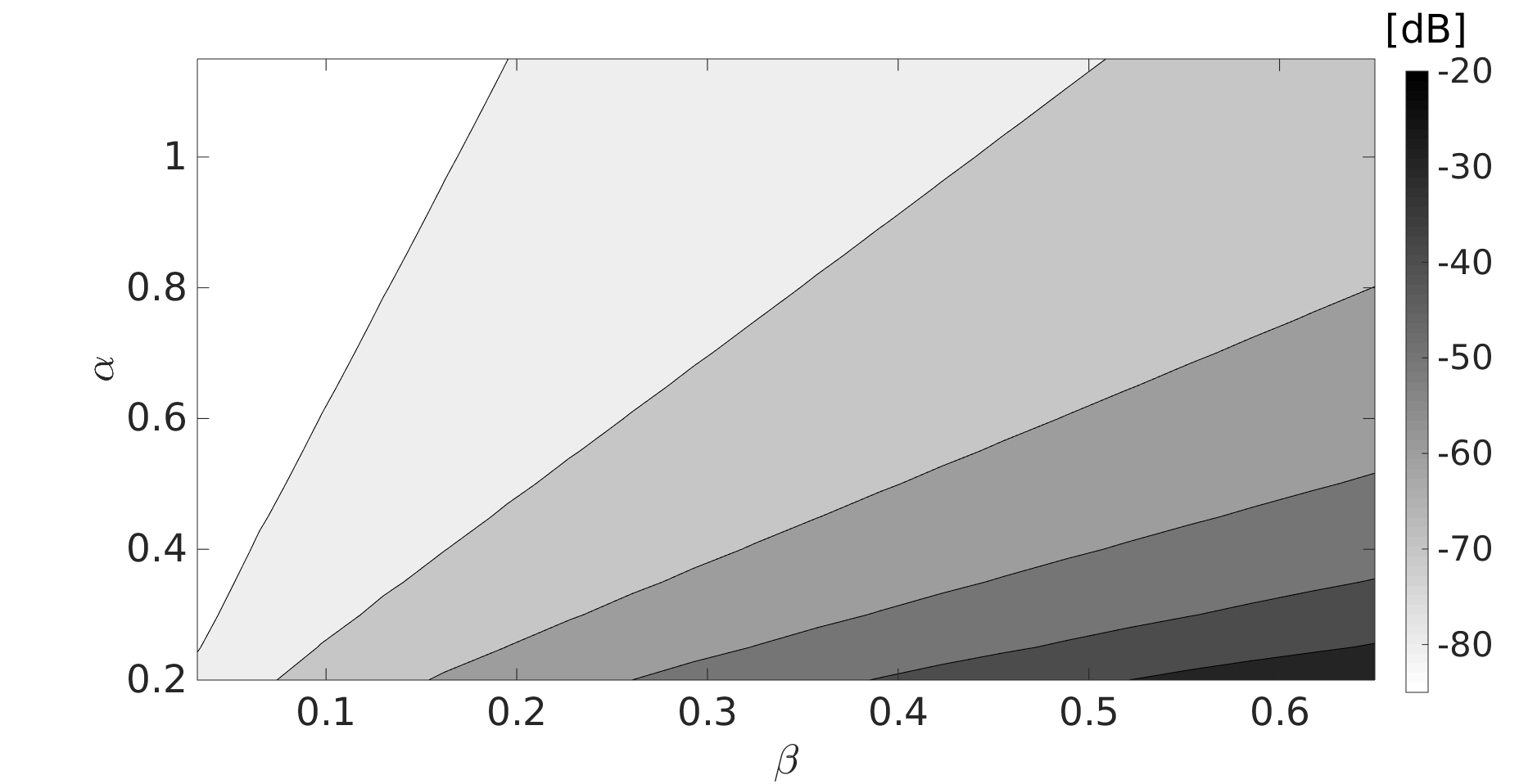} %\\(a)\\
	\end{tabular}
	\end{center}
	\caption{Color graphs where the $x$--axis is the $\beta$ coefficient and the $y$--axis represents the different coefficients $\alpha$. 
	Color white denotes low errors whereas dark regions mean high errors $\epsilon$. Two temporal steps are considered: 
	(Left) Fixing $\Delta t=10^{-3}$ seconds and (Right) fixing $\Delta t=10^{-4}$ seconds.}
	\label{alfaVsBeta}
	\end{figure}

	\textit{Third configuration} ($\alpha$ and $\beta$ with temporal--dependence). To conclude, we develop several experiments considering  
	time--dependent parameters. As mentioned at the beginning of this section, the chosen temporal profiles are based upon the papers 
	found in the literature. From \cite{2019Liumulti} and \cite{jehan2011planar} and by simplicity, some physical data have been modified.
	More precisely, we analyze the following two cases:
	\begin{itemize}
		\item \textit{Case $1$.}
		$\alpha(t)=5\cos\left(\displaystyle\frac{\pi t}{4}\right)~,~~~~~~~~~~~~~
		\beta(t)=\displaystyle\frac{1}{\cos\left(\displaystyle\frac{\pi t}{4}\right)}~~,~\forall t\in[0~1].$
		\item \textit{Case $2$.}
		$\alpha(t)=(t+1)^2~,~~~~~~~~~~~~~~~~\beta(t)=\displaystyle\frac{0.5}{t+1}~~~~~~~~~,~\forall t\in[0~1].$
	\end{itemize} 
	In both cases, we introduce these parameters into the implementation scheme (\ref{matrixrepre}) and also considering the initial condition 
	 (\ref{CI})  and the source term (\ref{CI2}). Recall that the $\alpha$ and $\beta$ profiles defined above explicitly appear in the source term (\ref{CI2}),
	 and whose solution (\ref{SCI}) is free of the dispersion and diffusion parameters. As mentioned, these specific conditions
	 make possible to define a benchmark (\ref{error}), which is either absolute or representative measure to fit the accuracy of these time--variation  
	 parameters into the numerical scheme (\ref{matrixrepre}). Therefore, we compute the error, $\epsilon$, in each case by using three different 
	 temporal samplings, $\Delta t=\{10^{-2},10^{-3},10^{-4}\}$ seconds.  Moreover, we define two measures  that give proper error bounds useful to 
	 easily calibrate the accuracy of the scheme when those time--dependent parameters are used. Therefore, we define both, the upper and the lower 
	 bound by considering the theoretical assumption  H\ref{h.th.error} (see theorem \ref{teo.error}) as well as the numerical results from 
	 Figure \ref{alfaVsBeta}. Briefly speaking,  the main idea consists in providing  a reliable error interval  that permits to establish a simple 
	 calibration of the method through simulations with constant  parameters. To do that,  we define the upper error, denoted by $\epsilon_{max}$,
	 and the lower error, denoted by $\epsilon_{min}$, by considering the constant pairs $(\bar\alpha,\bar\beta)$ that make maximum (resp.  minimum) 
	the error $\epsilon$. Note that these pairs $(\bar\alpha,\bar\beta)$ would strongly depend on the profile defined and their values would be different at 
	each case treated.  For example, if we consider the Case 1, the maximum error, $\epsilon_{max}$, is obtained when 
	$\bar\alpha=5/\sqrt{2}$ and $\bar\beta=\sqrt{2}$, whereas the minimum error, $\epsilon_{min}$ occurs for 
	$\bar\alpha=5$ and $\bar\beta=1$. For the Case 2,  $\epsilon_{max}$ is obtained with a simulation with  $\bar\alpha=1$ and 
	$\bar\beta=0.5$, whereas $\epsilon_{min}$ is achieved when  $\bar\alpha=4$ and $\bar\beta=0.25$.

\begin{table*}[h]
  \begin{center}
     \caption{Analysis of the error, $\epsilon$ and the upper and lower-bound error, $\epsilon_{max}$ and $\epsilon_{min}$ respectively. 
The three quantities are measured considering $1$ second of time simulation and $N=32$ nodes.}
    \begin{tabular}{|c||c|c|c||c||c|c|c|}
	\hline
      Case 1  & $\Delta t=10^{-4}$ & $\Delta t=10^{-3}$ & $\Delta t=10^{-2}$ & Case 2 & $\Delta t=10^{-4}$ & $\Delta t=10^{-3}$ & $\Delta t=10^{-2}$ \\
	\hline\hline
       $\epsilon_{max}$    & 0.00070345 & 0.00220118  & 0.00744720 & $\epsilon_{max}$ & 0.00027205 & 0.00087034 & 0.00322214  \\
	\hline
       $\epsilon_{min}$ & 0.00018275  & 0.00059578 & 0.00251925 & $\epsilon_{min}$ &  0.00005478 & 0.00019861 & 0.00134123  \\
      \hline
       $\epsilon$   & 0.00029501 & 0.00095356 & 0.00341955 & $\epsilon$ & 0.00009874 & 0.00032937 & 0.00181271  \\
	\hline
    \end{tabular}\label{TABLAEST3}
  \end{center}
\end{table*}
	The $\epsilon$ results of both cases and their corresponding error bounds, $\epsilon_{max}$ and $\epsilon_{min}$, are recorded in Table 
	\ref{TABLAEST3} considering three different temporal steps. It is worth pointing out that the results exhibit in Table \ref{TABLAEST3} are in 
	full concordance with the previous explanations and the error, $\epsilon$, is always within the error interval $[\epsilon_{min}, \epsilon_{max}]$, 
	no matter which temporal step is employed.
 
\section{Conclusions}
	The linear KdVB equation with non--periodic boundary conditions and time--dependent coefficients has been numerically 
	analyzed using the LPG method for the spatial discretization, and a finite difference scheme for the temporal behavior.
	The core of our analysis are new estimates related to the stability and convergence problems for this kind of equations, which now 
	involve non constant coefficients. Specifically, the convergence result proved in Theorem \ref{teo.error}
	shown a trade off between dispersion and diffusion parameters added into the model with a view to establishing  upper estimates in the form 
	$C_{\alpha}(N^{-r}+(\Delta t)^{2})+C_{\beta}(N^{1-r}+(\Delta t)^{2})$, where $C_\alpha\approx (\min\alpha(t))^{-1}$,
	$C_\beta\approx \|\beta\|_{L^\infty([0,T])}$ and $r\geq 2$. Respect to the stability, it can swing sharply if the dispersion 
	coefficient $\alpha(t)$ is small enough, see Theorem \ref{th1.stability}. 

	The computed results of the KdVB equation \eqref{sys.kdvb.linear}  exhibit the high accuracy of the proposed method based 
	in the Euler scheme--LPG approximation.  First, we have define a proper benchmark \eqref{error}, which is based on the data
	\eqref{CI}--\eqref{CI2}, providing evidences of the theoretical results in the temporal and spatial convergence. Moreover, we complement the 
	analysis of the spatial convergence pointing out the relation between  the numerical error and the projected source term, $\mathbf{CF}$.
	Secondly,  we develop a wide numerical analysis of the $\alpha(t)$ and $\beta(t)$  parameters. Finally, we present several experiments that show the 
	best performance ranges and the proper calibration of time--dependent  profiles.
		
	 To the best of our knowledge, our framework constitutes a first approach for studying from a numerical point of view
	 dynamic systems of odd--order dispersive equations with coefficients that can vary in time. 
	 The results obtained in this paper shown the role and trade off of these temporal parameters into the model.

	Finally, there are many ways to extend these ideas to future works. For example, a first future work could consider nonlinearity 
	into the model \eqref{sys.kdvb.linear}  with its respective coefficient, and it is associated to the convective term  
	$\gamma(t)uu_x$, with $\gamma(t)\neq\alpha(t)$,\, $\gamma(t)\neq\beta(t)$. Nevertheless, it worth mentioning that the 
	nonlinear analysis requires additional techniques that are not considered in this work.
	We invite readers to review the references  for more details on the nonlinear case.
	 For this kind of systems, a details comparison of the LPG approach to other methods in terms 
	of accuracy and efficiency would also be of interest. In a more general sense, similar ideas may also be fruitful 
	in considering evolution dynamics described by PDEs including time parametric dependencies located in its coefficients.

%% The Appendices part is started with the command \appendix;
%% appendix sections are then done as normal sections
%% \appendix

%\section{Acknowledgments}

%% \label{}

%% If you have bibdatabase file and want bibtex to generate the
%% bibitems, please use
%%
%%  \bibliographystyle{elsarticle-num} 
%%  \bibliography{<your bibdatabase>}

%% else use the following coding to input the bibitems directly in the
%% TeX file.

%\begin{thebibliography}{00}

%% \bibitem{label}
%% Text of bibliographic item

%\bibliographystyle{plain}
%\bibliography{Biblio}

\begin{thebibliography}{40}
%
\bibitem{2012Bhrawy}
Ali~H. Bhrawy and M.~M. Al-Shomrani.
\newblock A {J}acobi dual-{P}etrov {G}alerkin-{J}acobi collocation method for
  solving {K}orteweg-de {V}ries equations.
\newblock {\em Abstr. Appl. Anal.}, pages Art. ID 418943, 16, 2012.

\bibitem{bona1985travelling}
J.~L. Bona and M.~E. Schonbek.
\newblock Travelling-wave solutions to the {K}orteweg-de {V}ries-{B}urgers
  equation.
\newblock {\em Proc. Roy. Soc. Edinburgh Sect. A}, 101(3-4):207--226, 1985.

\bibitem{burgers1940application}
J.~M. Burgers.
\newblock Application of a model system to illustrate some points of the
  statistical theory of free turbulence.
\newblock {\em Nederl. Akad. Wetensch., Proc.}, 43:2--12, 1940.

\bibitem{bookCanuto}
Claudio Canuto, M.~Yousuff Hussaini, Alfio Quarteroni, and Thomas~A. Zang.
\newblock {\em Spectral methods in fluid dynamics}.
\newblock Springer Series in Computational Physics. Springer-Verlag, New York,
  1988.

\bibitem{2007Caruntu}
Dumitru~I. Caruntu.
\newblock Classical {J}acobi polynomials, closed-form solutions for transverse
  vibrations.
\newblock {\em J. Sound Vibration}, 306(3-5):467--494, 2007.

\bibitem{2009DengMa}
Zhenguo Deng and Heping Ma.
\newblock Optimal error estimates of the {F}ourier spectral method for a class
  of nonlocal, nonlinear dispersive wave equations.
\newblock {\em Appl. Numer. Math.}, 59(5):988--1010, 2009.

\bibitem{2009Doha1}
E.~H. Doha, W.~M. Abd-Elhameed, and A.~H. Bhrawy.
\newblock Efficient spectral ultraspherical-{G}alerkin algorithms for the
  direct solution of {$2n$}th-order linear differential equations.
\newblock {\em Appl. Math. Model.}, 33(4):1982--1996, 2009.

\bibitem{2009Doha3}
E.~H. Doha, A.~H. Bhrawy, and W.~M. Abd-Elhameed.
\newblock Jacobi spectral {G}alerkin method for elliptic {N}eumann problems.
\newblock {\em Numer. Algorithms}, 50(1):67--91, 2009.

\bibitem{2011Doha}
E.~H. Doha, A.~H. Bhrawy, and R.~M. Hafez.
\newblock A {J}acobi-{J}acobi dual-{P}etrov-{G}alerkin method for third- and
  fifth-order differential equations.
\newblock {\em Math. Comput. Modelling}, 53(9-10):1820--1832, 2011.

\bibitem{2009Doha2}
Eid~H. Doha and Ali~H. Bhrawy.
\newblock A {J}acobi spectral {G}alerkin method for the integrated forms of
  fourth-order elliptic differential equations.
\newblock {\em Numer. Methods Partial Differential Equations}, 25(3):712--739,
  2009.

\bibitem{2018Fangetal}
Jinwei Fang, Boying Wu, and Wenjie Liu.
\newblock An explicit spectral collocation method for the linearized
  {K}orteweg--de {V}ries equation on unbounded domain.
\newblock {\em Appl. Numer. Math.}, 126:34--52, 2018.

\bibitem{feng2007traveling}
Zhaosheng Feng and Roger Knobel.
\newblock Traveling waves to a {B}urgers-{K}orteweg-de {V}ries-type equation
  with higher-order nonlinearities.
\newblock {\em J. Math. Anal. Appl.}, 328(2):1435--1450, 2007.

\bibitem{1982Fenton}
J.~D. Fenton and M.~M. Rienecker.
\newblock A {F}ourier method for solving nonlinear water-wave problems:
  application to solitary-wave interactions.
\newblock {\em J. Fluid Mech.}, 118:411--443, 1982.

\bibitem{1978Fornberg}
B.~Fornberg and G.~B. Whitham.
\newblock A numerical and theoretical study of certain nonlinear wave
  phenomena.
\newblock {\em Philos. Trans. Roy. Soc. London Ser. A}, 289(1361):373--404,
  1978.

\bibitem{gao2015variety}
Xin-Yi Gao.
\newblock Variety of the cosmic plasmas: general variable-coefficient
  {K}orteweg-de {V}ries-{B}urgers equation with experimental/observational
  support.
\newblock {\em EPL (Europhysics Letters)}, 110(1):15002, 2015.

\bibitem{2019Gao}
Xin-Yi Gao.
\newblock Mathematical view with observational/experimental consideration on
  certain {$(2+1)$}-dimensional waves in the cosmic/laboratory dusty plasmas.
\newblock {\em Appl. Math. Lett.}, 91:165--172, 2019.

\bibitem{1977Orszag}
David Gottlieb and Steven~A. Orszag.
\newblock {\em Numerical analysis of spectral methods: theory and
  applications}.
\newblock Society for Industrial and Applied Mathematics, Philadelphia, Pa.,
  1977.
\newblock CBMS-NSF Regional Conference Series in Applied Mathematics, No. 26.

\bibitem{1997David}
David Gottlieb and Chi-Wang Shu.
\newblock On the {G}ibbs phenomenon and its resolution.
\newblock {\em SIAM Rev.}, 39(4):644--668, 1997.

\bibitem{2007Goubet}
Olivier Goubet and Jie Shen.
\newblock On the dual {P}etrov-{G}alerkin formulation of the {K}d{V} equation
  on a finite interval.
\newblock {\em Adv. Differential Equations}, 12(2):221--239, 2007.

\bibitem{2014Hannonlinear}
Jiu-Ning Han, Jun-Hua Luo, and Jun-Xiu Li.
\newblock Nonlinear electrostatic coherent structures: solitary and shock waves
  in a dissipative, nonplanar multi-component quantum plasma.
\newblock {\em Astrophysics and Space Science}, 349(1):305--315, 2014.

\bibitem{hussain2011korteweg}
S~Hussain and S~Mahmood.
\newblock Korteweg-de vries burgers equation for magnetosonic wave in plasma.
\newblock {\em Physics of Plasmas}, 18(5):052308, 2011.

\bibitem{jehan2011planar}
Nusrat Jehan, Arshad~M Mirza, and M~Salahuddin.
\newblock Planar and cylindrical magnetosonic solitary and shock waves in
  dissipative, hot electron-positron-ion plasma.
\newblock {\em Physics of Plasmas}, 18(5):052307, 2011.

\bibitem{2010Korkmaz}
Alper Korkmaz.
\newblock Numerical algorithms for solutions of {K}orteweg-de {V}ries equation.
\newblock {\em Numer. Methods Partial Differential Equations},
  26(6):1504--1521, 2010.

\bibitem{korteweg1895xli}
D.~J. Korteweg and G.~de~Vries.
\newblock On the change of form of long waves advancing in a rectangular canal,
  and on a new type of long stationary waves.
\newblock {\em Philos. Mag. (5)}, 39(240):422--443, 1895.

\bibitem{2000Li}
Jian Li, Heping Ma, and Weiwei Sun.
\newblock Error analysis for solving the {K}orteweg-de {V}ries equation by a
  {L}egendre pseudo-spectral method.
\newblock {\em Numer. Methods Partial Differential Equations}, 16(6):513--534,
  2000.

\bibitem{2019Liumulti}
Jian-Guo Liu, Wen-Hui Zhu, Li~Zhou, and Yao-Kun Xiong.
\newblock Multi-waves, breather wave and lump--stripe interaction solutions in
  a (2+1)-dimensional variable-coefficient {K}orteweg--de {V}ries equation.
\newblock {\em Nonlinear Dynamics}, 97(4):2127--2134, 2019.

\bibitem{2000Maetal}
Heping Ma and Weiwei Sun.
\newblock A {L}egendre-{P}etrov-{G}alerkin and {C}hebyshev collocation method
  for third-order differential equations.
\newblock {\em SIAM J. Numer. Anal.}, 38(5):1425--1438, 2000.

\bibitem{2001Maetal}
Heping Ma and Weiwei Sun.
\newblock Optimal error estimates of the {L}egendre-{P}etrov-{G}alerkin method
  for the {K}orteweg-de {V}ries equation.
\newblock {\em SIAM J. Numer. Anal.}, 39(4):1380--1394, 2001.

\bibitem{1988Quarteroni}
Y.~Maday and A.~Quarteroni.
\newblock Error analysis for spectral approximation of the {K}orteweg-de
  {V}ries equation.
\newblock {\em RAIRO Mod\'{e}l. Math. Anal. Num\'{e}r.}, 22(3):499--529, 1988.

\bibitem{philips1992data}
Wilfried Philips and Geert De~Jonghe.
\newblock Data compression of ecg's by high-degree polynomial approximation.
\newblock {\em IEEE transactions on biomedical engineering}, 39(4):330--337,
  1992.

\bibitem{2020QinMa}
Yonghui Qin and Heping Ma.
\newblock Legendre-tau-{G}alerkin and spectral collocation method for nonlinear
  evolution equations.
\newblock {\em Appl. Numer. Math.}, 153:52--65, 2020.

\bibitem{1994ShenJie}
Jie Shen.
\newblock Efficient spectral-{G}alerkin method. {I}. {D}irect solvers of
  second- and fourth-order equations using {L}egendre polynomials.
\newblock {\em SIAM J. Sci. Comput.}, 15(6):1489--1505, 1994.

\bibitem{2003ShenJie}
Jie Shen.
\newblock A new dual-{P}etrov-{G}alerkin method for third and higher odd-order
  differential equations: application to the {K}d{V} equation.
\newblock {\em SIAM J. Numer. Anal.}, 41(5):1595--1619, 2003.

\bibitem{bookGabor}
Gabor Szeg\"{o}.
\newblock {\em Orthogonal {P}olynomials}.
\newblock American Mathematical Society, New York, 1939.
\newblock American Mathematical Society Colloquium Publications, v. 23.

\bibitem{tchiotsop2007ecg}
Daniel Tchiotsop, Didier Wolf, Valerie Louis-Dorr, and Rene Husson.
\newblock Ecg data compression using jacobi polynomials.
\newblock In {\em 2007 29th Annual International Conference of the IEEE
  Engineering in Medicine and Biology Society}, pages 1863--1867. IEEE, 2007.

\bibitem{2014Triki}
Houria Triki and Abdul-Majid Wazwaz.
\newblock Traveling wave solutions for fifth-order {K}d{V} type equations with
  time-dependent coefficients.
\newblock {\em Commun. Nonlinear Sci. Numer. Simul.}, 19(3):404--408, 2014.

\bibitem{2019WangPengzhan}
Pengfei Wang and Pengzhan Huang.
\newblock Convergence of the {C}rank-{N}icolson extrapolation scheme for the
  {K}orteweg--de {V}ries equation.
\newblock {\em Appl. Numer. Math.}, 143:88--96, 2019.

\bibitem{yu2011solitonic}
Xin Yu, Yi-Tian Gao, Zhi-Yuan Sun, and Ying Liu.
\newblock Solitonic propagation and interaction for a generalized
  variable-coefficient forced {K}orteweg--de {V}ries equation in fluids.
\newblock {\em Physical Review E}, 83(5):056601, 2011.

\bibitem{2010Yuan}
Juan-Ming Yuan and Jiahong Wu.
\newblock A dual-{P}etrov-{G}alerkin method for two integrable fifth-order
  {K}d{V} type equations.
\newblock {\em Discrete Contin. Dyn. Syst.}, 26(4):1525--1536, 2010.
\end{thebibliography}

\end{document}